\documentclass{amsart}

\usepackage[utf8]{inputenc}
\usepackage{fullpage}
\usepackage[colorlinks,linkcolor=blue,citecolor=blue]{hyperref}
\usepackage{graphicx}
\usepackage{amsmath}
\usepackage{amsthm}
\usepackage{amsfonts}
\usepackage{amssymb}
\usepackage{bbm}
\usepackage{mathtools}
\usepackage{extarrows}
\usepackage{array}   
\newcolumntype{L}{>{$}l<{$}}

\usepackage{tikz-cd}
\usepackage{pgf,tikz}
\usetikzlibrary{
  knots,
  hobby,
  decorations.pathreplacing,
  shapes.geometric,
  calc
}

\usepackage{mathrsfs}
\usetikzlibrary{arrows}

\newtheorem{Thm}{Theorem}[section]

\newtheorem{Cor}[Thm]{Corollary}
\newtheorem{prop}[Thm]{Proposition}
\newtheorem{lem}[Thm]{Lemma}

\newtheorem{defn}[Thm]{Definition}

\theoremstyle{definition}

\theoremstyle{remark}

\newtheorem{rmk}{Remark}[section]

\usepackage{eucal}
\usepackage{graphicx}
\usepackage{multirow}
\usepackage[all,knot]{xy}

\DeclareMathOperator{\Mod}{Mod}

\DeclareMathOperator{\PGL}{PGL}

\DeclareMathOperator{\SMod}{SMod}
\DeclareMathOperator{\LMod}{LMod}
\DeclareMathOperator{\PMod}{PMod}
\DeclareMathOperator{\End}{End}

\DeclareMathOperator{\Irr}{Irr}
\DeclareMathOperator{\id}{id}
\DeclareMathOperator{\Id}{Id}
\DeclareMathOperator{\obj}{obj}
\DeclareMathOperator{\Obj}{Obj}
\calclayout
\newcommand{\RNum}[1]{\uppercase\expandafter{\romannumeral #1\relax}}

\newcommand{\BZ}{\mathbb{Z}}
\newcommand{\Zn}[1]{\BZ/{#1}\BZ}

\usepackage{graphicx}
\graphicspath{ {./graph/} }

\newcommand{\BB}{\mathcal{B}}

\renewcommand{\AA}{\mathcal{A}}
\newcommand{\CC}{{\mathcal{C}}}

\newcommand{\Hom}{\operatorname{Hom}}

\allowdisplaybreaks[3]
\begin{document}
\title{Braiding structures on categorical multi-Interval Jones-Wassermann subfactor}
\author{Zhengwei Liu and Yuze Ruan}
\address{Z. LIU, Yau Mathematical Sciences Center and Department of Mathematics, Tsinghua University, Beijing, 100084, China}
\address{Yanqi Lake Beijing Institute of Mathematical Sciences and Applications, Huairou District,
Beijing, 101408, China}
\email{liuzhengwei@mail.tsinghua.edu.cn}

\address{Y. Ruan, Yau Mathematical Sciences Center and Department of Mathematics, Tsinghua University, Beijing, 100084, China}
\email{yuzeruan@mail.tsinghua.edu.cn}

\begin{abstract}
 In this paper, we construct braiding structures on the multi-interval Jones-Wassermann subfactor planar algebra associated with any unitary modular fusion category. Utilizing this construction, we provide a new proof of the self-duality of these subfactors. Furthermore, we demonstrate that these braidings induce a projective unitary representation of the balanced superelliptic mapping class group; consequently, these structures effectively encode the non-trivial higher-genus data of the underlying category. As an application of this correspondence, we derive a generalized Verlinde formula as $2$-box Fourier duality of the planar algebra.  
\end{abstract}
\maketitle

\section{Introduction}
The relationship between conformal field theory (CFT) and tensor categories is a central topic in representation theory and mathematical physics. A foundational result establishes that the representation category of a nice CFT yields a modular fusion category (MFC) \cite{Hua05, KLX05}. The converse problem—the reconstruction program—asks whether every MFC arises as the representation category of a CFT. This question was initiated by Vaughan Jones \cite{Jones17} and remains open. Direct algebraic reconstruction is difficult, and current work largely focuses on specific families of categories, such as twisted quantum doubles and Tambara–Yamagami categories \cite{EG22, EG23}.

An alternative perspective is guided by the principle that if the reconstruction program holds, any concrete construction or invariant on the CFT side should possess a purely category-theoretic counterpart. In \cite{LX19}, Liu and Xu realized this by constructing the analogue of Jones–Wassermann subfactors and their underlying planar algebras entirely within the framework of unitary modular fusion categories (UMFCs), providing a categorical generalization of the classic conformal net constructions \cite{LR95, Was98, Xu00, KLM01, KLX05}. By defining a categorical Fourier transformation and proving its invertibility, they established the self-duality of these Jones–Wassermann subfactors using only the data of the underlying category. Moreover, these subfactors are closely related to the permutation orbifold \cite{KLX05}. Following this philosophy, considering a categorical analogue of the permutation orbifold construction becomes a natural step for the reconstruction program. In category theory, this analogue is called permutation gauging \cite{Muger05, CGPW16}, which is formulated via abstract categorical extensions and equivariantization \cite{ENO05, ENO10}. 

However, on the CFT side, concrete constructions of permutation orbifolds \cite{Bantay98, BHS98, Bantay02, KLM01, LX04, KLX05} provide a comparatively explicit picture.  In particular, they suggest a close relationship between the genus-zero data of the permutation-orbifold CFT and the higher-genus data of the original CFT, as expressed by the twisted/untwisted correspondence; see \cite{BS11, Gui21} for topological and geometric interpretations.  Furthermore, the modular $S$-matrix of the cyclic permutation orbifold, can be computed explicitly from the modular data of the original theory \cite{Bantay98, KLX05, BHS98, DRX21, DXY22}, where the formulas depend only on the original $S$- and $T$-matrices.  In contrast, the categorical theory of permutation gauging remains largely existence-theoretic. While \cite{GJ19} proves the existence of the gauged category, by computing the cohomological obstructions developed in \cite{ENO10}, the detailed structures of the gauged theory and their precise relation to the original category remain elusive.

In our previous work \cite{LR24}, we provided an explicit $\Zn2$-permutation gauging for arbitrary modular fusion categories, establishing a concrete correspondence between the braidings of the gauged theory and the higher-genus symmetric mapping class group representations derived from the RT-TQFT of the original category. A central ingredient is the construction of certain braiding structures on the $2$-interval Jones–Wassermann subfactor planar algebra associated with the original category. Therefore, to generalize this construction to cyclic permutation gaugings of arbitrary order $n$, the essential prerequisite is the construction of corresponding braiding structures on the multi-interval Jones–Wassermann subfactor planar algebra.

In this paper, we construct explicit braiding structures on the $m$-interval Jones–Wassermann subfactors planar algebra associated with any UMFCs for $m>2$. These operators enable us to characterize the Fourier transformation, a key operation in \cite{LX19} for the construction of self-dual Jones–Wassermann subfactors, as a composition of braiding operators, thereby providing a new proof of their self-duality (Theorem \ref{thm:Fourier_transform}). More importantly, we demonstrate that these operators induce a projective representation of the balanced superelliptic mapping class group (Theorem \ref{thm:Smod_rep}). The group consists of lifts of parity preserving or parity reversing braids under the balanced cyclic branched cover of degree $m$ \cite{GW17lifting}, this contrast with the $m=2$ case, where the entire braid group admits such a lift \cite{BH73,MW21}. Consequently, the complex higher-genus symmetries of the original UMFC can be systematically analyzed through planar tangles and isotopy identities. As an application we determine the structural constants of the $2$-box convolution product  for multi-interval Jones–Wassermann subfactor planar algebra in terms of a generalized verlinde formula (Theorem \ref{thm:Verlinde_formula}). 

The remainder of this paper is organized as follows. In Section $2$, we provide preliminaries on unitary modular fusion categories, with an emphasis on graphical calculus notations and identities. In Section $3$, we introduce the configuration space and construct the braiding operators acting upon it; we then establish various relations among these operators, providing a new proof of the self-duality of multi-interval Jones–Wassermann subfactors. In Section $4$, we demonstrate that these operators induce a projective unitary representation of the balanced superelliptic mapping class group. Finally, in Section $5$, we explore the interaction between these braiding operators and other algebraic structures, and ultimately derive the structural constants of the 2-box convolution product via a generalized Verlinde formula.

\section*{Acknowledgement}
This work was supported by Beijing Municipal Science \& Technology Commission [Z221100002722017 to Z.L. and Y.R]; Beijing National Science Foundation Key Programs [Z220002 to Z.L.]; China’s National Key R\&D Programmes [2020YFA0713000 to Z.L.].

\section{Preliminaries}
We assume the reader is familiar with the theory of unitary modular fusion categories and the associated graphical calculus. For a detailed treatment, we refer to \cite{BK01, EGNO, Row05f, Tur10}.

Let $\CC$ be a unitary modular fusion category, We begin by fixing our notation and reviewing the essential graphical calculus identities within $\CC$. Throughout this paper, morphisms are represented as string diagrams read from top to bottom.

\subsection{Notations}
\begin{itemize}
    \item $\overline{\ \ }$: the dual functor, $\bar{X},\ \bar{f}$,
    \item $\dagger$:  Involutive antilinear contravariant endofunctor from the unitary structure,   
    \item $d_X$: quantum dimension of the object $X$,
    \item $\mu$: global dimension of $\CC$,  $\mu=\dim(\CC)=\sum_{V\in \Irr(\CC)}d^2_V$,
    \item $\delta$: the positive square root of $\mu$,
    \item $\Omega$: the Kirby colour $\sum_{V\in \Irr(\CC)}d_V V$,
    \item $\theta^{\pm}_X$: the twist for $X$,
    \item $p^{\pm}$: $p^{\pm}=\sum_{V\in \Irr(\CC)}\theta^{\pm}_Vd_V^2$, $p^+p^-=\mu$,
    \item $\eta$: $\eta=\frac{p^+}{\delta},\  \eta^{-1}=\frac{p^-}{\delta}$.
     \item ONB(X): an orthonormal basis of $\Hom_{\CC}(1,X)$ for object $X\in \Obj(\CC
     )$. When the corresponding morphism space is clear from context, we simply denote it by ONB.
    \end{itemize}

\subsection{Graphic calculus notations}
\begin{itemize}
    \item We employ red-colored loops to indicate that the corresponding strand is decorated by the Kirby color $\Omega$. Conversely, a red cup or cap denotes a coloring by the object $\bigoplus_{V\in \Irr(\CC)}d^{\frac{1}{2}}_V V$, and a red vertical strand denotes a coloring by the object $\bigoplus_{V\in \Irr(\CC)} V$. 
    \[
    \includegraphics[width=350pt]{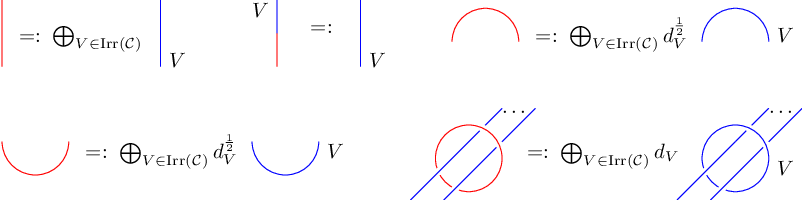}
    \]
    \item We utilize Frobenius reciprocity to construct a basis for $\Hom(X_1\cdots X_k,Y_1\cdots Y_k)$ from a basis of $\Hom(\mathbbm{1},Y_1\cdots Y_k\overline{X_k}\cdots \overline{X}_1)$:
    \[
    \includegraphics[width=320pt]{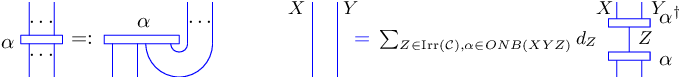}
    \]
    \item We adopt the same convention for $\Theta_1$, $\Theta_2$,$\Theta_{\CC}$ as in \cite{LX19} and \cite{LR24}
    \[
    \includegraphics[width=430pt]{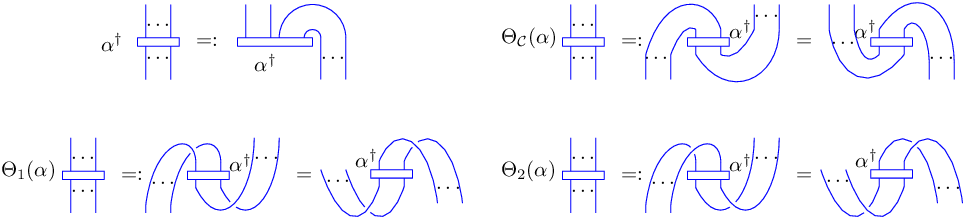}
    \]
    \item In the graphic calculus, similar to that in \cite{LR24}, we extensively use the dotted red loops to indicate their positions and simplify calculations (the underlying identities are justified by the cutting property).
    \begin{equation}\label{eq:cutting_identity}
    \includegraphics[width=400pt]{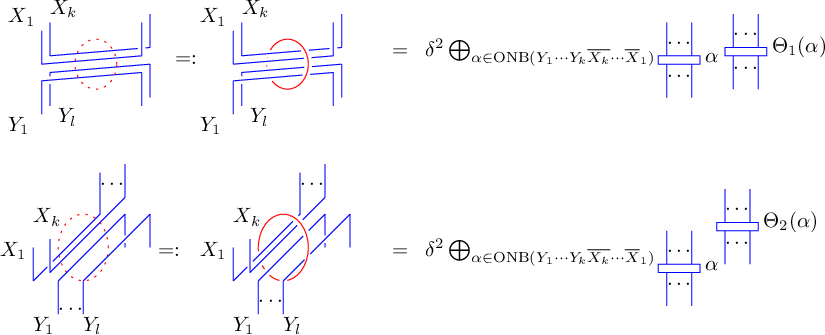}
    \end{equation}
    \item We adopt the same convention for contraction and inclusion:
    \[
    \includegraphics[width=150pt]{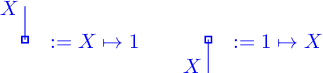}
    \]
    \item We employ the following notations about twists:
\[    
    \includegraphics[width=350pt]{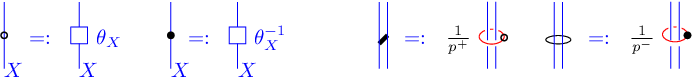}
\]
\end{itemize}

\subsection{Graphic calculus identities}

\begin{itemize}
    \item Twist property: $\theta_{X\otimes Y}=c_{Y,X}c_{X,Y}\theta_X\otimes\theta_Y$.
    \[
    \includegraphics[width=140pt]{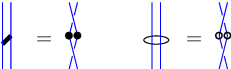}
    \]
    \item Cutting property of $\Omega$ (cf.~\eqref{eq:cutting_identity}). 
    \[
    \includegraphics[width=100pt]{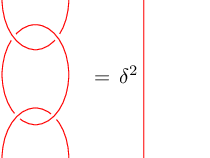}
    \]
    \item Handle slide property of $\Omega$.
    \[
    \includegraphics[width=200pt]{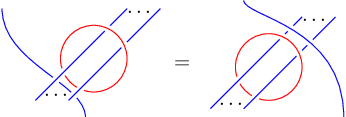}
    \]
\end{itemize}

\section{Braiding operators on the configuration space}
In this section, we review the definition of the configuration spaces and their associated basic operations. Subsequently, we construct the braiding operators and establish various algebraic relations among them, providing a new proof of the self-duality of multi-interval Jones–Wassermann subfactors.

\subsection{Configuration space}
\begin{defn}[{\cite[Sec.~2]{LX19}}]\label{def:cfbasis} Let $\CC$ be a unitary modular fusion category, $X_{i,j},Y_{i,j}\in \obj(\CC)$, we define $X_{-1,j}=X_{m-1,j}:=\mathbbm{1}$, $Y_{0,j}=X_{0,j},Y_{m-1,j}=\overline{X_{m-2,j}}$. Let $a_i\in \Hom_{\CC}(\mathbbm{1}, \bigotimes^{n-1}_{j=0} Y_{i,j})$, $b_{i,j}\in \Hom(\mathbbm{1},\overline{X_{i-1,j}}X_{i,j}\overline{Y_{i,j}})$ and $b_{0,j}=\id_{Y_{0,j}}, b_{m-1,j}=\id_{Y_{m-1,j}}$.
The configuration space $Conf(\CC)_{m,n}$ is the space spanned by the vectors as in Figure \ref{fig:cfbasis}. 
\begin{figure}[h]
    \centering
    \includegraphics[width=0.8\linewidth]{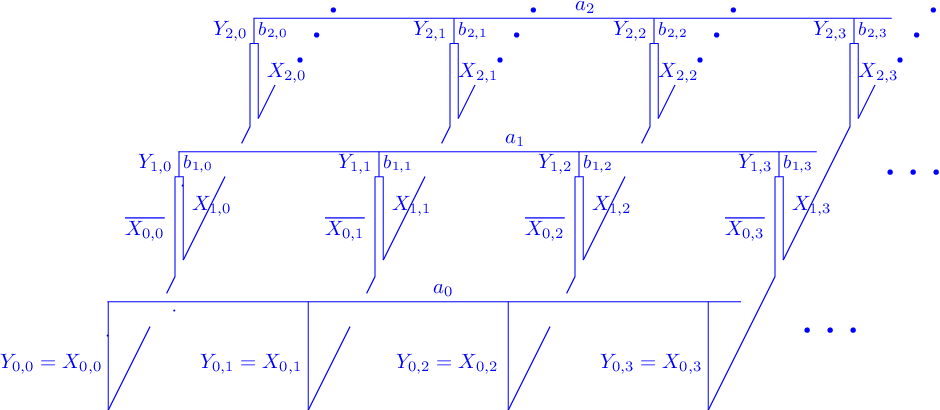}
    \caption{Vectors in configuration space}
    \label{fig:cfbasis}
\end{figure}
\end{defn}

The configurations in $X$ and $Y$ directions are given by the morphisms:
\[
\mathbf{a}_i:=\bigotimes^{n-1}_{j=0} b_{i,j}\circ a_i,\ \ \mathbf{b}_j:=\bigotimes^{m-2}_{i=0} ev_{X_{i,j}}\circ\bigotimes^{m-1}_{i=0}b_{i,j},
\]
which are depicted in Figure \ref{fig:XYConf}.

\begin{figure}
    \centering
    \includegraphics[width=1\linewidth]{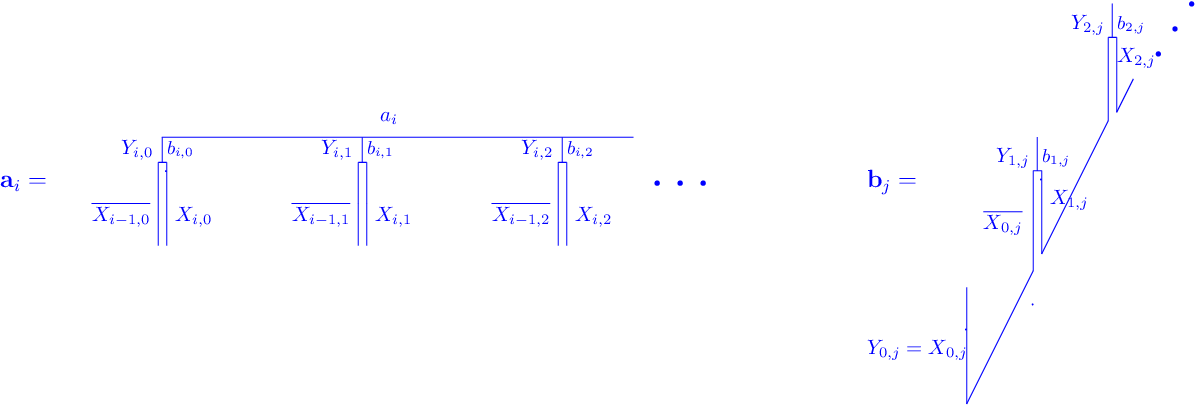}
    \caption{Configurations in $X$ and $Y$ directions}
    \label{fig:XYConf}
\end{figure}

\begin{rmk}
Our definition is a refinement of that in \cite[Sec.~2]{LX19}, enabling a more delicate description of morphisms and simplifying the associated graphical calculus.
\end{rmk}

The inner product on $Conf(\CC)_{m,n}$ is induced from the inner products on the spaces $\Hom_{\CC}(\mathbbm{1}, \bigotimes^{n-1}_{j=0} Y_{i,j})$ and $\Hom_{\CC}(\mathbbm{1},\overline{X_{i-1,j}}X_{i,j}\overline{Y_{i,j}})$. Consequently, the collection of vectors formed by  $a_i\in ONB(\bigotimes^{n-1}_{j=0} Y_{i,j})),\ b_{i,j}\in ONB(\overline{X_{i-1,j}}X_{i,j}\overline{Y_{i,j}})$ and the scaling factor $\prod_{1\leq  i\leq m-2,j}\sqrt{d_{Y_{i,j}}}$ constitutes an orthonormal basis of $Conf(\CC)_{m,n}$, denoted by $ONB(Conf(\CC)_{m,n})$.

Recall the action of $\Theta_2,\Theta_\CC$ on morphism spaces defined in \cite{LX19}, here we define the action of $\tilde{\Theta}_2$ on $\mathbf{a}_i$ as follows, see Figure \ref{fig:Theta_2def}.
\begin{figure}
    \centering
    \includegraphics[width=1\linewidth]{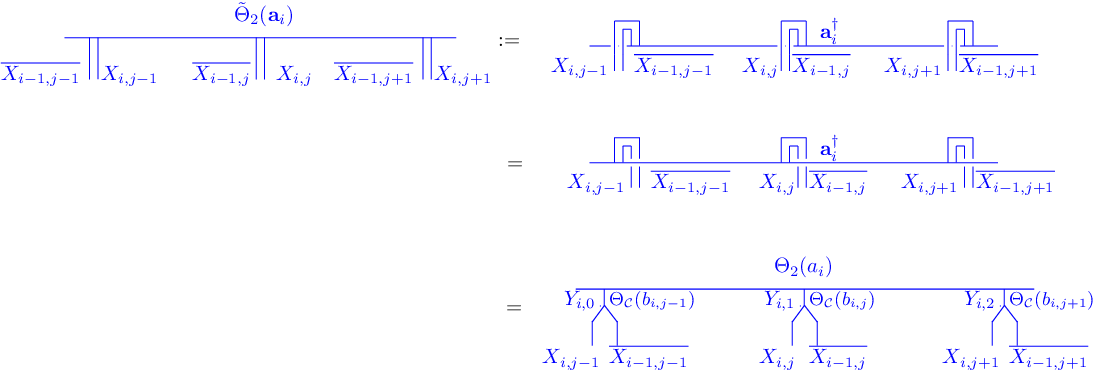}
    \caption{The $\Theta_2$ action on $\mathbf{a}_i$.}
    \label{fig:Theta_2def}
\end{figure}

we have the following lemma,
\begin{lem}
We have the following identities:
\[
\begin{aligned}
 \Theta_2(\mathbf{a}_i)=&\bigotimes^{n-1}_{j=0}c^{-1}_{X_{i,j},\overline{X_{i-1,j}}} \tilde{\Theta}_2(\mathbf{a}_i)=\bigotimes^{n-1}_{j=0}c^{-1}_{X_{i,j},\overline{X_{i-1,j}}} \Theta_{\CC}(b_{i,j})\circ\Theta_2(a_i),\\
 \Theta_{\CC}(\mathbf{b}_{j})=&\bigotimes^{m-2}_{i=0} ev_{\overline{X_{m-2-i,j}}}\circ\bigotimes^{m-1}_{i=0}\Theta_\CC(b_{m-1-i,j}).
\end{aligned}
\]
\end{lem}

\begin{proof}
The identities follow from straightforward graphic calculus.  

\end{proof}
\begin{rmk}
Similarly to $\Theta_2$, We define the action $\tilde{\Theta}_2$ on $Conf(\CC)_{m,n}$ by reflecting the indices in the $Y$-direction.      
\end{rmk}

\begin{defn}
The contraction and inclusion maps defined in \cite{LX19} admit the following graphical interpretation, as illustrated in Figure \ref{fig:Contrac_inclu}.

\begin{figure}
    \centering
    \includegraphics[width=1\linewidth]{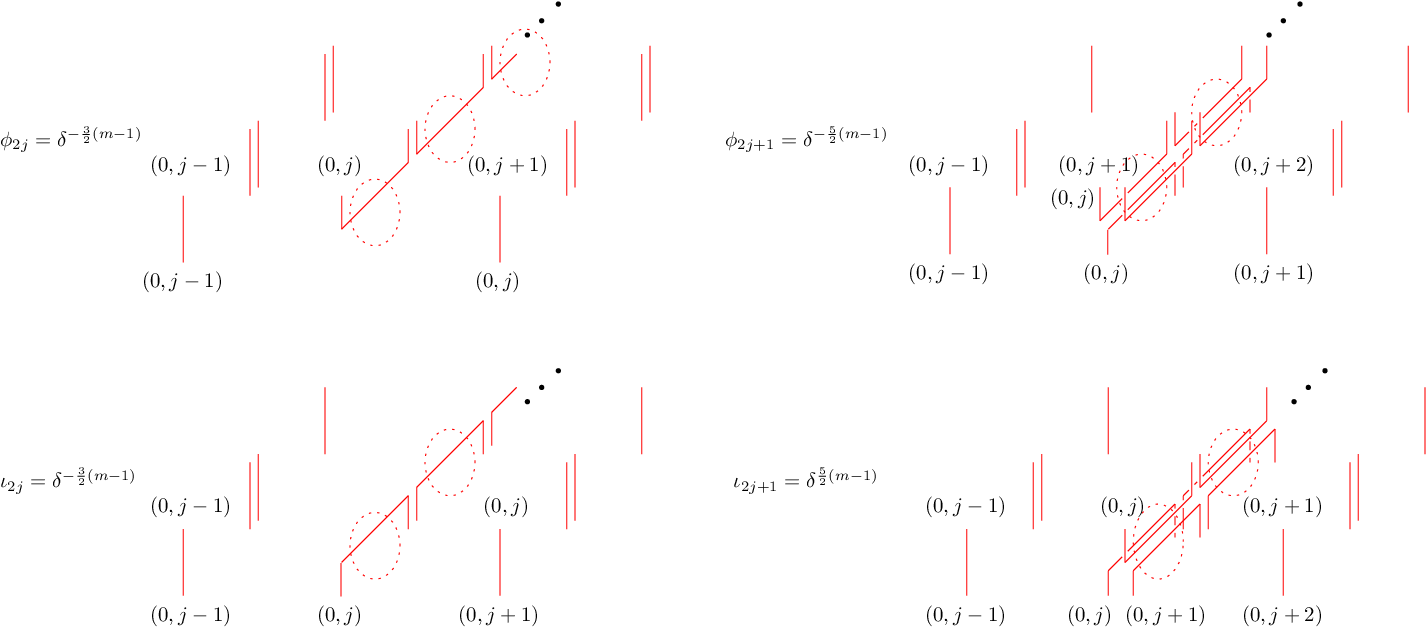}
    \caption{Contraction and inclusion}
    \label{fig:Contrac_inclu}
\end{figure}    
\end{defn}

\begin{defn}
The rotation operators $\rho_1$ and $\rho_2$ are defined as in \cite[Sec.~3]{LX19}, where $\rho_1$ represents a clockwise rotation of $\frac{2\pi}{n}$ about the $Y$-direction, and $\rho_2$ represents a clockwise rotation of $\frac{2\pi}{m}$ about the $X$-direction."   
\end{defn}

\subsection{Braiding operators}
We first define the unitary operator $u\in \End(Conf(\CC)_{m,n})$. 
\begin{defn}
 We define the action of $u$ on $\mathbf{a}_i$ and the vectors in $Conf(\CC)_{m,n}$ as illustrated in Figure \ref{fig:u_def1}, \ref{fig:u_def2}.
 \begin{figure}
     \centering  \includegraphics[width=1\linewidth]{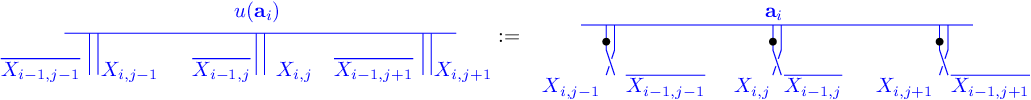}
     \caption{Definition of the action $u$ on $\mathbf{a}_i$}
     \label{fig:u_def1}
 \end{figure}
 \begin{figure}
     \centering
    \includegraphics[width=1.1\linewidth]{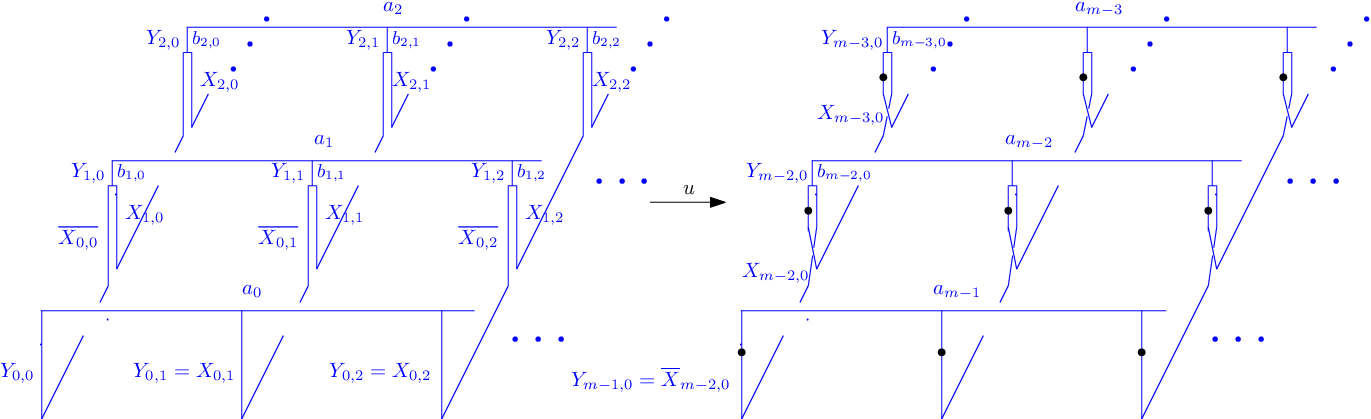}
     \caption{Definition of the action $u$ on $Conf(\CC)_{m,n}$}
     \label{fig:u_def2}
 \end{figure}
\end{defn}

\begin{lem}
 The following identities are immediate:
\begin{align}
u^2(x)&=\prod_{i,j}\theta^{-1}_{Y_{i,j}}x,\\
\Theta_2u&=u^{-1}\Theta_2,    \\
\Theta_2\rho_2&=\rho_2^{-1}\Theta_2.
\end{align}
\end{lem}

\begin{defn}
 The operators $T_{i,j}$ are defined in Figure \ref{fig:ElementaryBraidings}, with their construction depending on the parity of $i$. One can show their unitarity by performing simple graphical calculus.

\begin{figure}[h]
    \centering
    \includegraphics[width=1\linewidth]{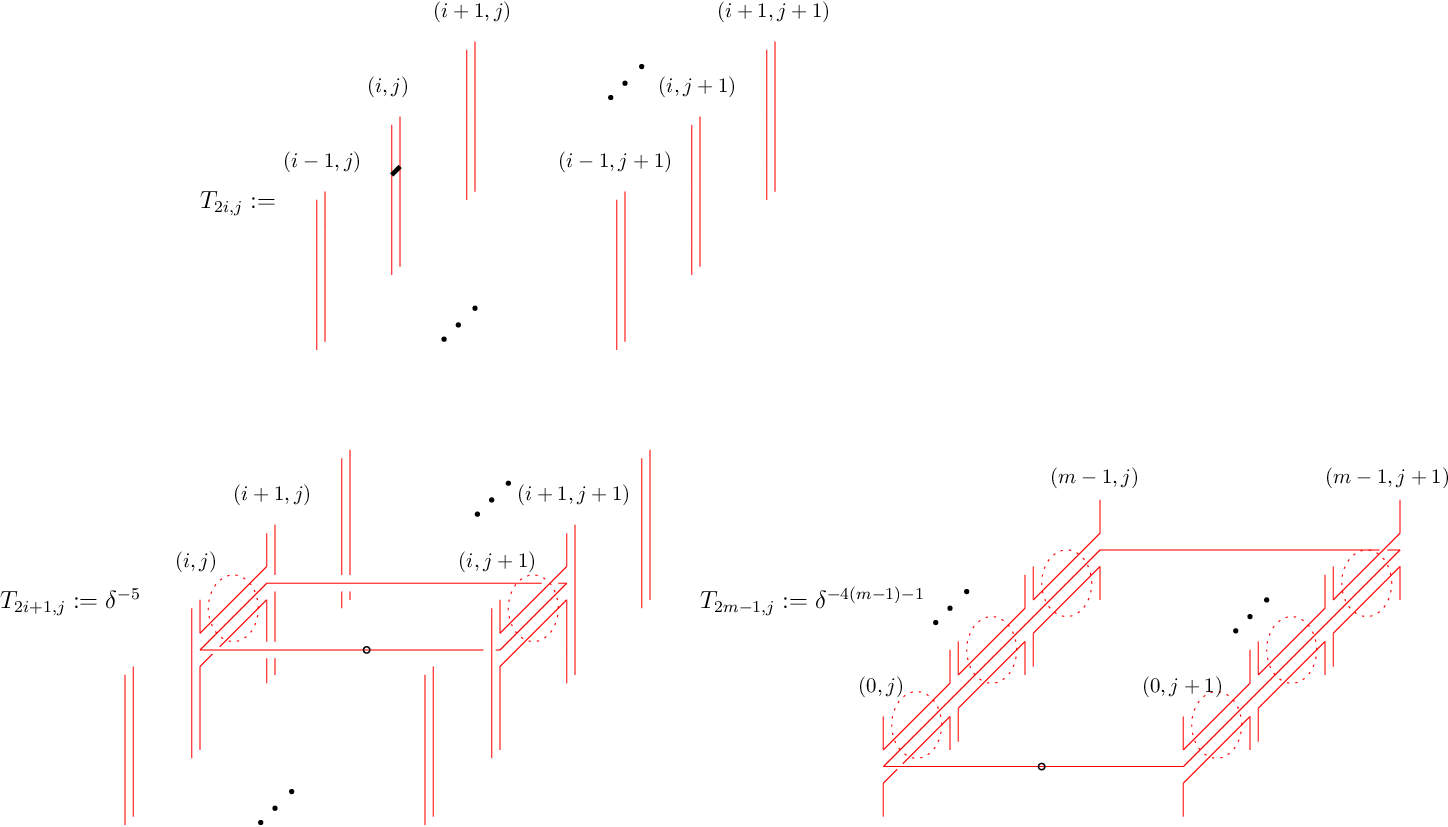}
    \caption{Operators $T_{i,j}$}
    \label{fig:ElementaryBraidings}
\end{figure}
\end{defn}

\begin{prop}\label{prop:T_rels}
The operators $T_{i,j}\in \End(Conf(\CC)_{m,n})$ satisfy following relations:
\begin{align*}
T_{2i+1,j}T_{k,l}T_{2i+1,j}&=T_{k,l}T_{2i+1,j}T_{k,l},&&for \ (k,l)\in \{(2i,j),(2i,j+1),(2i+2,j),(2i+2,j+1)\}\\
T_{2m-1,j}T_{k,l}T_{2m-1,j}&=T_{k,l}T_{2m-1,j}T_{k,l},&&for \ (k,l)\in \{(2m-2,j),(2m-2,j+1),(0,j),(0,j+1)\} \\
T_{i,j}T_{k,l}&=T_{i,j}T_{k,l}. &&otherwise
\end{align*}
\end{prop}
\begin{proof}
The proof follows from the following graphical calculus. The commuting relations for $T_{2m-1,j}$ follow from the handle slide property applied to the vertical red circles, while the remaining commuting relations are straightforward. For the braid relations, it suffices to verify, for example, the following identity:
\begin{equation}\label{eq:braidrelation}
T_{2i+1,j}T_{2i+2,j+1}T_{2i+1,j}=\eta T_{2i+2,j+1}T_{2i+1,j}T_{2i+2,j+1}.    
\end{equation}

First, we simplify the product $T_{2i+1,j}T_{2i+2,j+1}T_{2i+1,j}$ as illustrated in Figure \ref{fig:Proofbraidrel1} (note we omit the scalar factor $\sqrt{X_{i,j}X_{i,j+1}}$ for clarity): the second equality follows from the cutting property, the third from the properties of the configuration space, and applying the required twists.

\begin{figure}[h]
    \centering
    \includegraphics[width=0.9\linewidth]{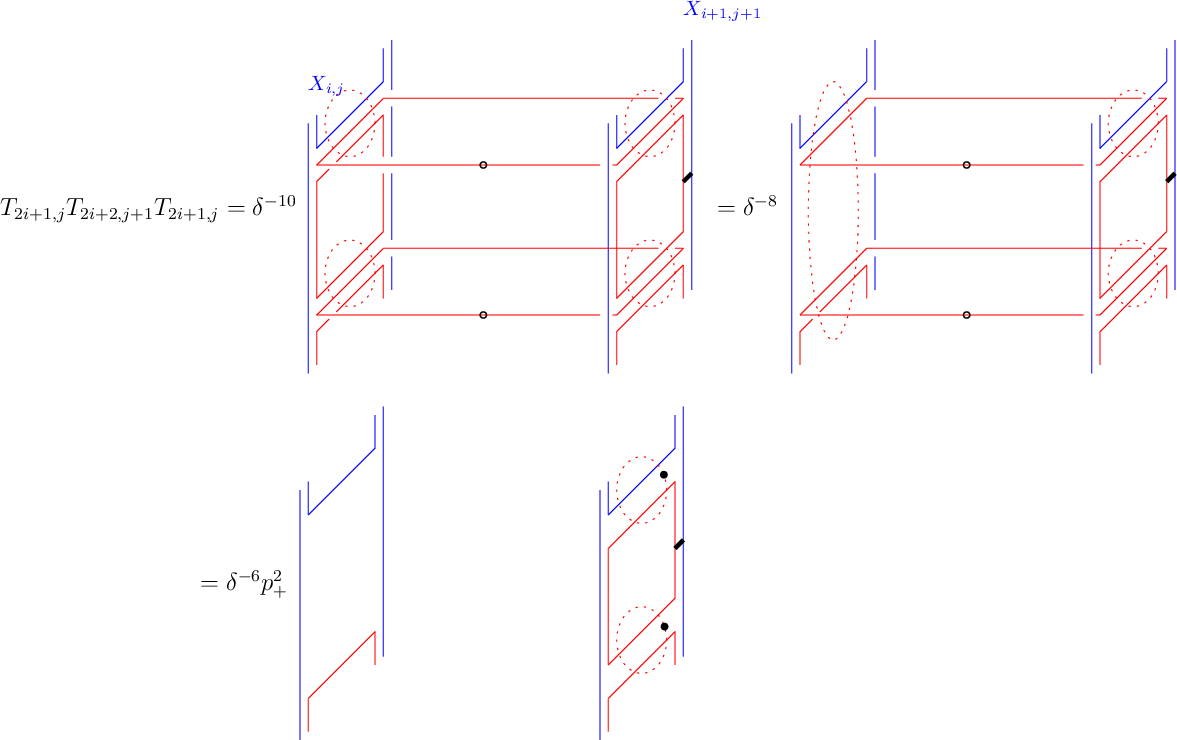}
    \caption{Simplification for $T_{2i+1,j}T_{2i+2,j+1}T_{2i+1,j}$}
    \label{fig:Proofbraidrel1}
\end{figure}

Next, we focus on the right-hand portion of the diagram and perform the graphical calculus shown in Figure \ref{fig:Proofbraidrel2}, repeatedly invoking the twist property.  

\begin{figure}[h]
    \centering
    \includegraphics[width=0.87\linewidth]{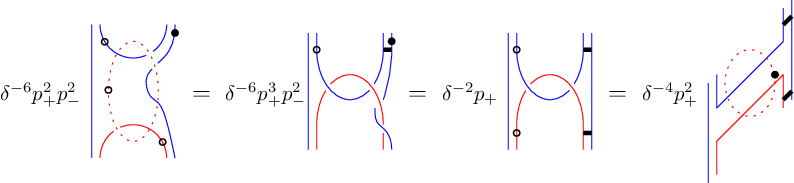}
    \caption{Graphic Calculus}
    \label{fig:Proofbraidrel2}
\end{figure}

Finally, by comparing the resulting diagrams and coefficients, we arrive at the equality (\ref{eq:braidrelation}). 

\begin{figure}[h]
    \centering
    \includegraphics[width=0.9\linewidth]{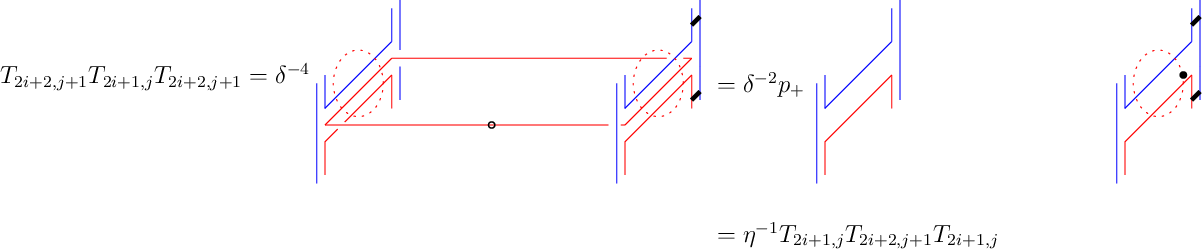}
    \caption{Proof for $T_{2i+1,j}T_{2i+2,j+1}T_{2i+1,j}=\eta T_{2i+2,j+1}T_{2i+1,j}T_{2i+2,j+1}$}
    \label{fig:Proofbraidrel3}
\end{figure}
\end{proof}

\begin{defn}
We define the following two unitary operators. Since the individual factors commute according to Proposition \ref{prop:T_rels}, the order of the product is irrelevant: 
 \[
 T_j:=\prod_{i=0}^{m-1}T_{2i,j},\ \  T'_j:=\prod_{i=0}^{m-1}T_{2i+1,j}.
 \]    
\end{defn}

\begin{lem}\label{lem:ThetaT_rel}
 We have the following identities:
\begin{align}
 \Theta_2T_{2i,j}&=T^{-1}_{2m-2-2i,j}\Theta_2, \label{eq:ThetaT_even}\\   \Theta_2T_{2i+1,j}=T^{-1}_{2m-3-2i,j}&\Theta_2,\  \Theta_2T_{2m-1,j}=T^{-1}_{2m-1,j}\Theta_2.\label{eq:ThetaT_odd}
\end{align}
\end{lem}

\begin{proof}
The identity \eqref{eq:ThetaT_even} follows directly. Regarding \eqref{eq:ThetaT_odd}, we prove the first identity in Figure \ref{fig:ThetaT_rel}, and the second follows by a similar argument. 

\begin{figure}
    \centering \includegraphics[width=1.0\linewidth]{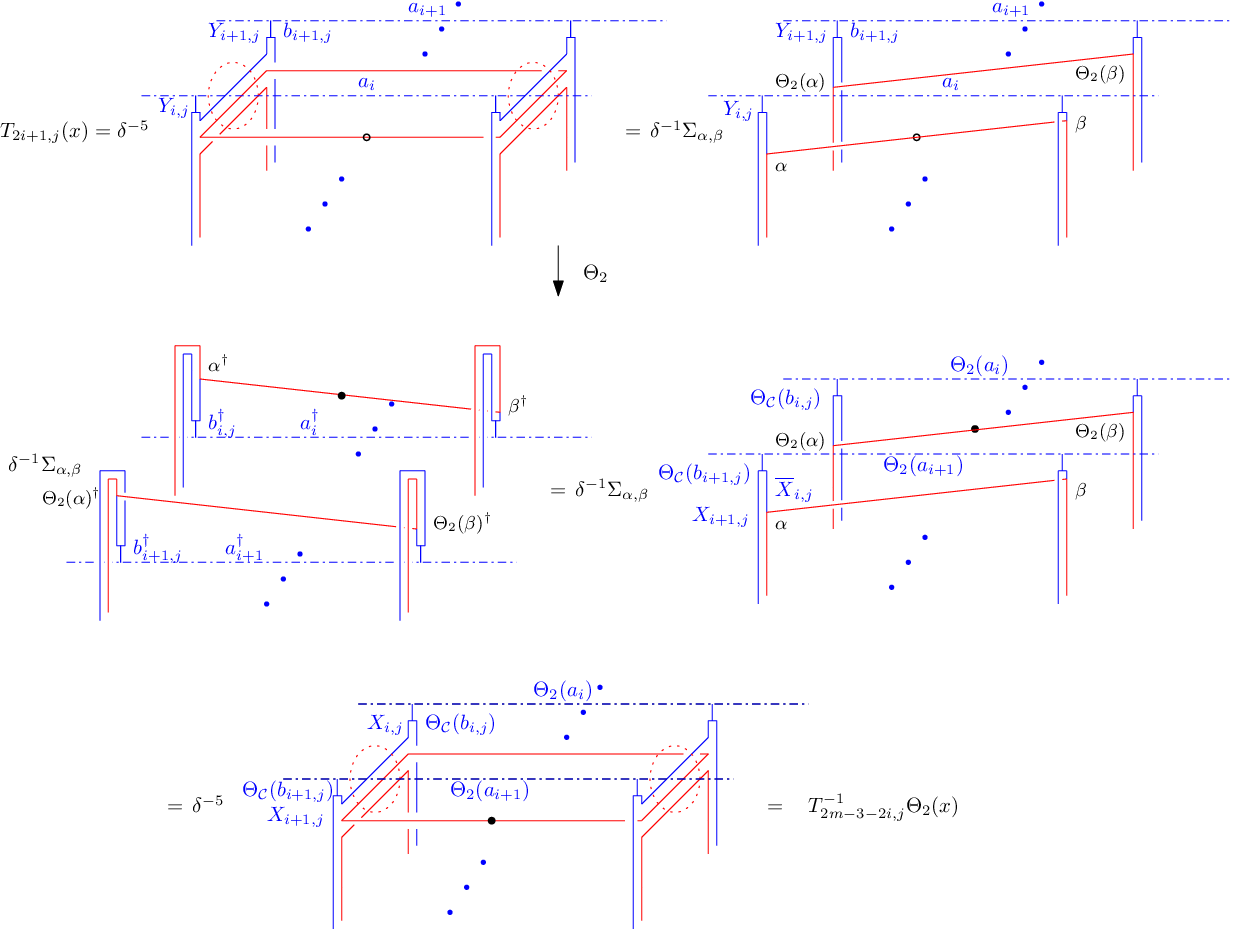}
    \caption{Relation between $T_{i,j}$ and $\Theta_2$}
    \label{fig:ThetaT_rel}
\end{figure}
\end{proof}

\begin{lem}\label{lem:uT_rels}
The following identites hold:
\begin{align}
u^{2}&=\prod^{n-1}_{j=1}T_j,\\
uT_{2i,j}u^{-1}&=T_{2m-2-2i,j},\\
uT_{2i+1,j}u^{-1}&=T_{2m-2-2i,j}T_{2m-4-2i,j+1}T_{2m-3-2i,j}T^{-1}_{2m-4-2i,j+1}T^{-1}_{2m-2-2i,j}\label{eq:uT_rels_2}, \\
uT_ju^{-1}&=T_j.
\end{align}
\end{lem}
\begin{proof}
The first two identities are straightforward. The third follows from the graphical calculus illustrated in Figure \ref{fig:uT_rel}, and the final identity is a direct consequence of the second.

\begin{figure}
    \centering
    \includegraphics[width=1\linewidth]{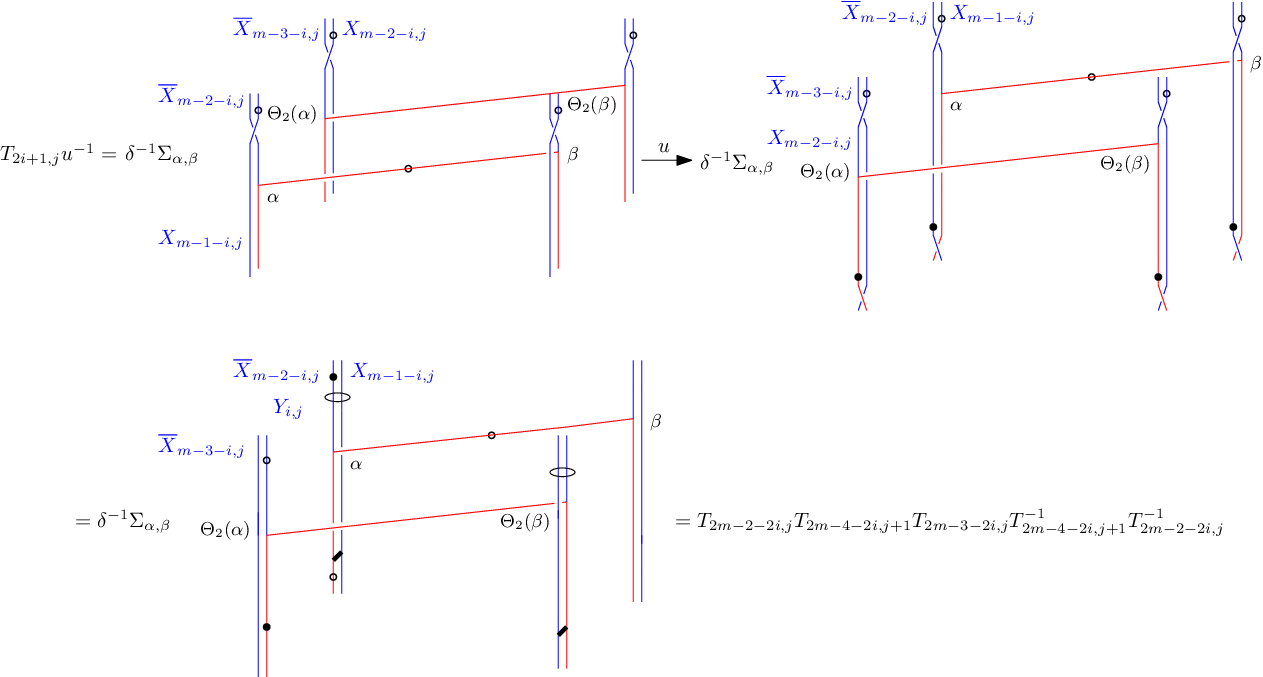}
    \caption{Relation between $u$ and $T_{2i+1,j}$}
    \label{fig:uT_rel}
\end{figure}
\end{proof}

\begin{defn}
We define the following two unitary operators,
\[
\AA_j:=\prod^{2m-2}_{i=0}T_{2m-2-i,j},\ \ \BB_{j}:=T_{0,j+1}\prod^{m-1}_{i=1}T_{2i-1,j}T_{2i,j+1}.
\] 
\end{defn}

The commutativity relations in Proposition \ref{prop:T_rels} give the following Lemma.
\begin{lem}\label{lem:ABTcomm_rels}
The following identities hold:
\begin{align}
[\AA_{j_1},T_{i,{j_2}}]&=0, \ \ (|j_1-j_2|\geq 2\ , \ j_2=j_1-1\ \text{and i even},\ j_1=j_2+1\ 
\text{and i odd})\\
[\BB_{j_1},T_{i,{j_2}}]&=0, \ \ (|j_1-j_2|\geq 2,\  j_2=j_1+1)\\
[\AA_{j_1},\AA_{j_2}]&=[\BB_{j_1},\BB_{j_2}]=0,\ (|j_1-j_2|\geq 2)\\
[\AA_{j_1},\BB_{j_2}]&=0.\ (|j_1-j_2|\geq 2,j_2=j_1+1)
\end{align}
\end{lem}

\begin{lem}\label{lem:utAB_rels}
We have the following identities:
\begin{align}
uT^{-1}_j\AA_ju^{-1}&=\BB_jT^{-1}_{j+1}, \\ 
u^{-1}\AA_jT^{-1}_ju&=T^{-1}_{j+1}\BB_j
\end{align}
     
\end{lem}
\begin{proof}
From Lemma \ref{lem:uT_rels}, we have: 
\[
\begin{aligned}
&T^{-1}_{2m-2-2i,j}uT_{2i+1,j}u^{-1}T_{2m-2-2i,j}=T_{2m-4-2i,j+1}T_{2m-3-2i,j}T^{-1}_{2m-4-2i,j+1},\\
\implies &uT^{-1}_{2i,j}T_{2i+1,j}T_{2i,j}u^{-1}=T_{2m-4-2i,j+1}T_{2m-3-2i,j}T^{-1}_{2m-4-2i,j+1}.
\end{aligned}
\]
Therefore, together with Lemma \ref{lem:ABTcomm_rels}, we have:
\[
\begin{aligned}
uT^{-1}_j\AA_ju^{-1}&=uT^{-1}_{2m-2,j}T_{2m-2,j}\prod_{i=m-2}^{0}T^{-1}_{2i,j}T_{2i+1,j}T_{2i,j}u^{-1}\\
&=\prod_{i=m-2}^{0}T_{2m-4-2i,j+1}T_{2m-3-2i,j}T^{-1}_{2m-4-2i,j+1}\\
&=(\prod_{i=0}^{m-2}T_{2i,j+1}T_{2i+1,j}T^{-1}_{2i,j+1})T_{2m-2,j+1}T^{-1}_{2m-2,j+1}\\
&=\BB_jT^{-1}_{j+1}     
\end{aligned}
\]
\end{proof}

\begin{prop}\label{prop:ABC_rels}
Let $C_{j,j+1}:=\Id\otimes(\boxtimes^{m-1}_{i=0}c^{-1}_{Y_{i,j},Y_{i,j+1}})\otimes \Id$ denote the braiding operator in $\CC^{\boxtimes m}$, the following equalities hold:
\begin{align}
\BB_j\AA_j&=T_{j+1}C_{j,j+1}=C_{j,j+1}T_j\\
\AA_j\BB_j&=T_{j}C_{j,j+1}=C_{j,j+1}T_{j+1}
\end{align}   
\end{prop}

\begin{proof}
We prove the first identity, noting that the proof of the second follows by a similar argument. The proof relies on the graphical calculus detailed in Figures \ref{fig:ABC_rel_1}, \ref{fig:ABC_rel_2}, \ref{fig:ABC_rel_3}. Note we omit the scalar factor $\prod_{i}\sqrt{X_{i,j}X_{i,{j+1}}}$ for simplicity.

In Figure \ref{fig:ABC_rel_1}, we apply the cutting property followed by the procedure established in \cite[Prop.~4.17]{LR24} to arrive at the desired form in the first layer. Next, we turn to the remaining two layers (Figure \ref{fig:ABC_rel_2}): we first isolate the dotted red loops in the upper-left and lower-right corners via handle slides, then apply the twist property to reach the final configuration. Finally, as shown in Figure \ref{fig:ABC_rel_3}, we perform a handle slide on one of the blue strands, followed by a cutting property application with respect to two solid red loops and one dotted red loop, to obtain the final result.

\begin{figure}
    \centering
    \includegraphics[width=0.65\linewidth]{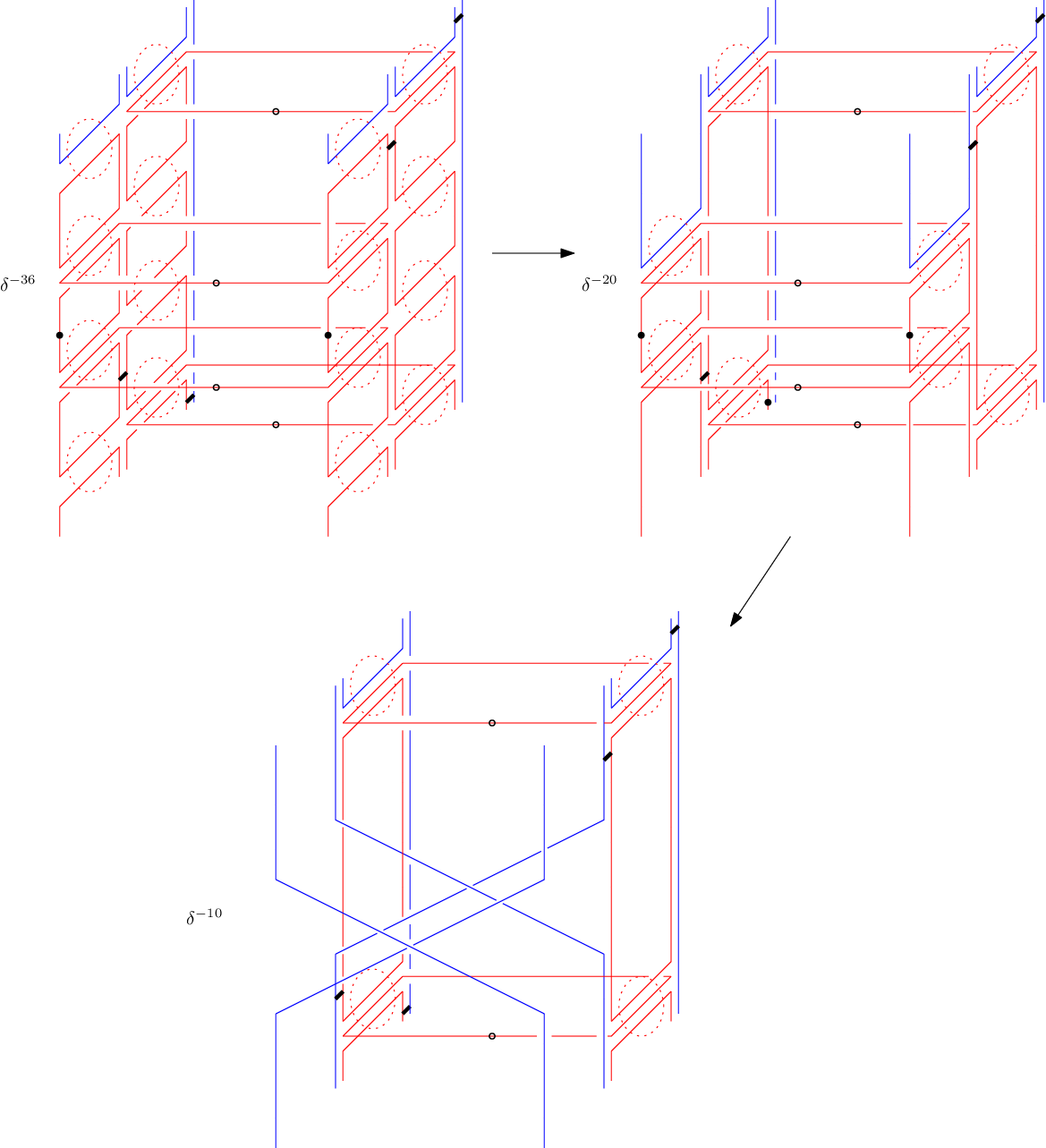}
    \caption{Proof of Proposition \ref{prop:ABC_rels}}
    \label{fig:ABC_rel_1}
\end{figure}

\begin{figure}
    \centering
    \includegraphics[width=0.65\linewidth]{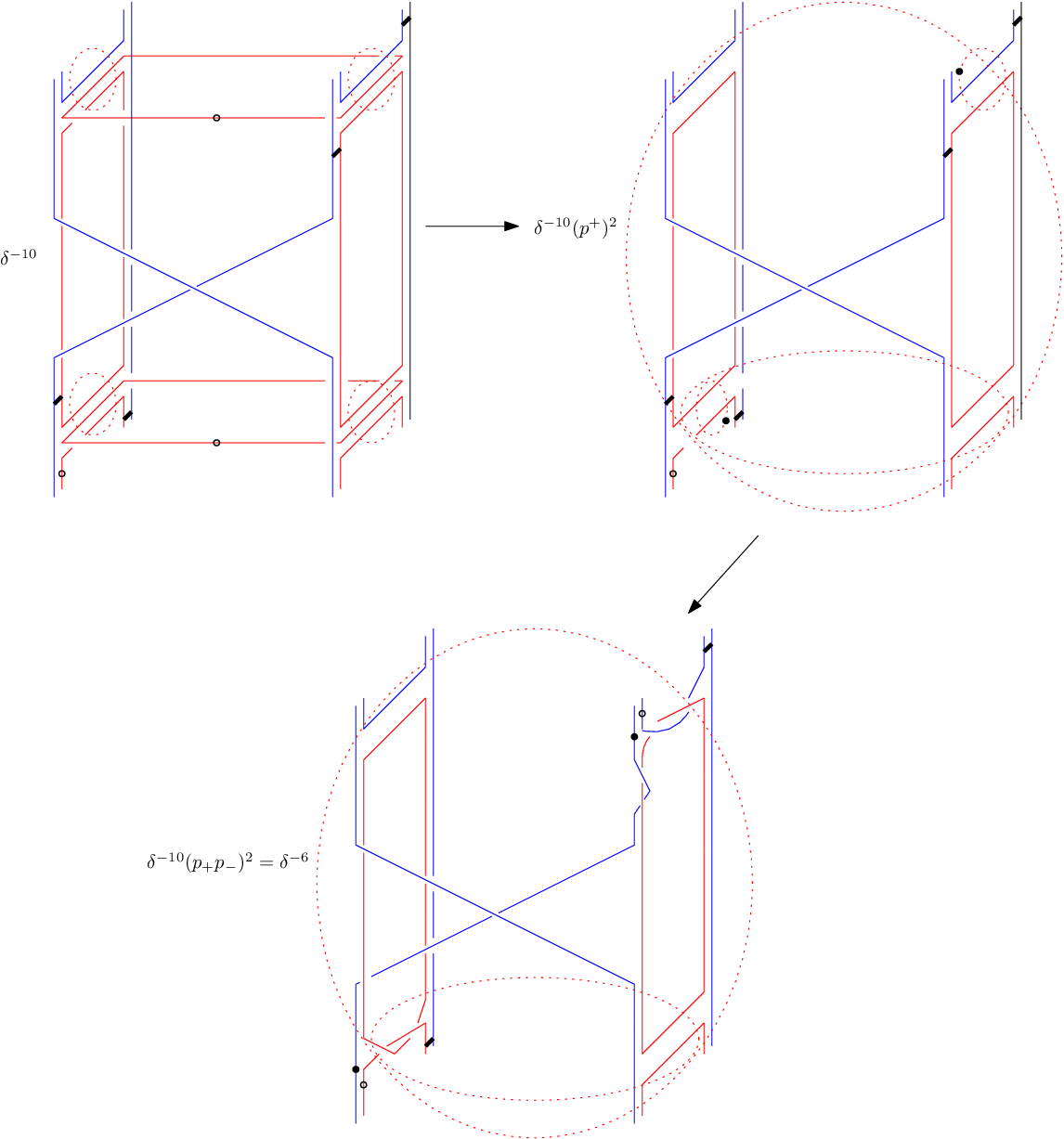}
    \caption{Proof of Proposition \ref{prop:ABC_rels}}
    \label{fig:ABC_rel_2}
\end{figure}

\begin{figure}
    \centering
    \includegraphics[width=0.65\linewidth]{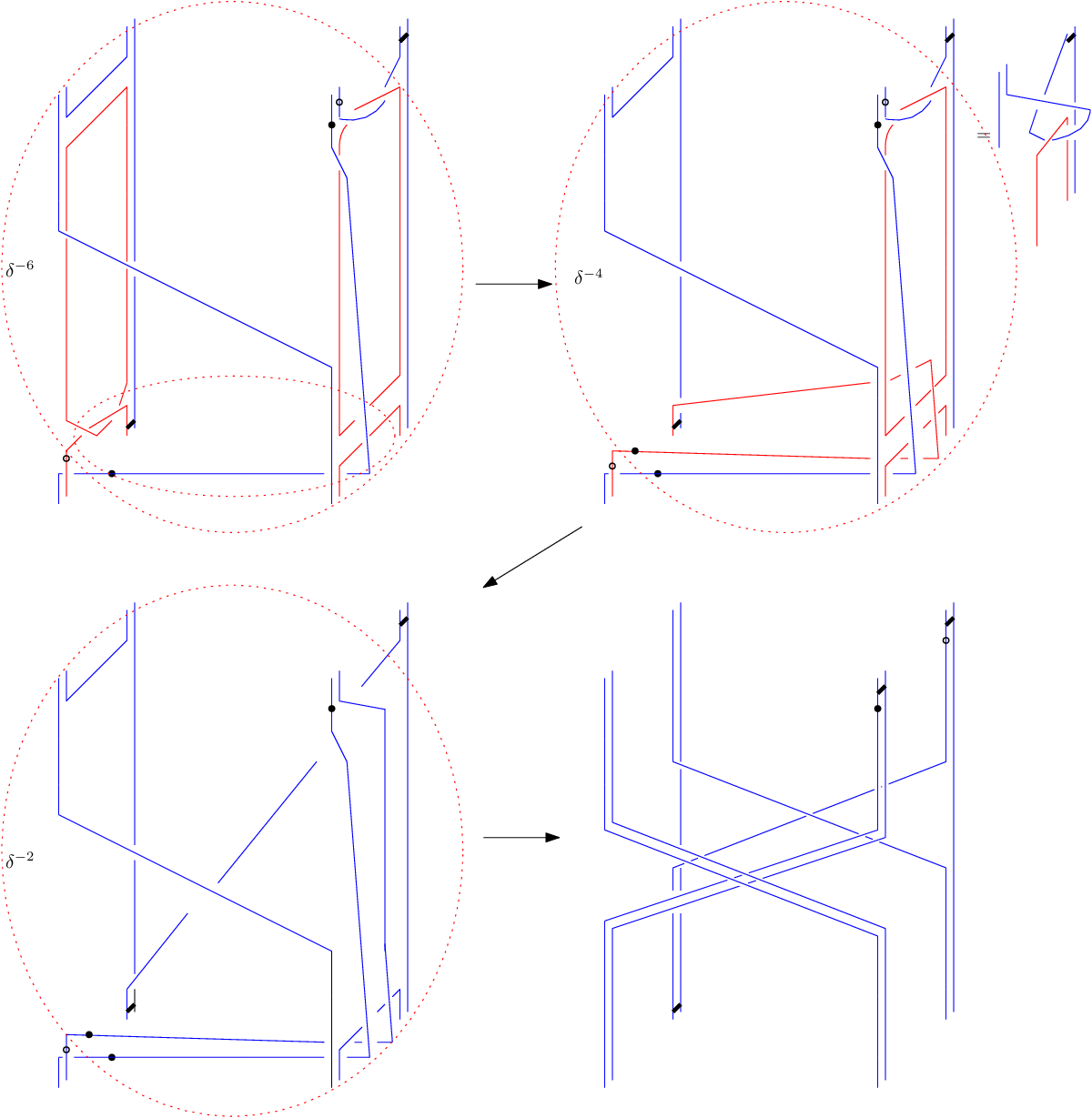}
    \caption{Proof of Proposition \ref{prop:ABC_rels}}
    \label{fig:ABC_rel_3}
\end{figure}
\end{proof}

As a direct consequence of the previous Proposition, we have the following identities:
\begin{align}
\BB_j\AA_jT^{\pm1}_{2i,j}=T^{\pm1}_{2i,j+1}\BB_j\AA_j,\ &\BB_j\AA_jT^{\pm1}_{2i,j+1}=T^{\pm1}_{2i,j}\BB_j\AA_j\label{eq:BAT_rels_1}\\ 
\BB_j\AA_jT^{-1}_{j}&=T^{-1}_j\AA_j\BB_j\label{eq:BAT_rels_2}
\end{align}
\begin{Cor}\label{Cor:ABTrho_1_rels}
We have the following equalities:
\begin{align}
(\prod^{n-2}_{j=0}\BB_{n-1-j}\AA_{n-1-j}T^{-1}_{j})\BB_{0}\AA_{0}&=\rho_1=\BB^{-1}_{n-1}\AA^{-1}_{n-1}\prod_{j=0}^{n-2}\BB^{-1}_{n-2-j}\AA^{-1}_{n-2-j}T_{n-2-j}, \\   
(\prod^{n-2}_{j=0}T^{-1}_J\AA_{j}\BB_{j})\AA_{n-1}\BB_{n-1}&=\rho^{-1}_1=\AA^{-1}_{0}\BB^{-1}_{0}\prod^{n-1}_{j=1}T_{j}\AA^{-1}_{j}\BB^{-1}_{j}.
\end{align}
\end{Cor}
\begin{proof}
 It follows from the equality:
 \[
 \begin{aligned}
  \rho_1=&(\prod^{0}_{j=n-2}C_{j,j+1})T_0=(\prod^{0}_{j=n-2}C^{-1}_{j,j+1})T^{-1}_0,\\
  =&T_{n-1}(\prod^{0}_{j=n-2}C_{j,j+1})=T^{-1}_{n-1}(\prod^{0}_{j=n-2}C^{-1}_{j,j+1}).
 \end{aligned}
 \]
\end{proof}
 In the next theorem, We demonstrate that the Fourier transform, as defined in \cite{LX19} via the pairing $LL$, can be decomposed into a product of the unitary operators introduced in this section.

\begin{Thm}\label{thm:Fourier_transform}
The Fourier transform is given by the following product of unitary operators:
\[
F=u^{-1}\prod^{n-2}_{j=0} T_{n-2-j}\AA^{-1}_{n-2-j},
\]
we have, for $x,x'\in ONB(Conf(\CC)_{m,n})$,
\[
<Fx, x'>=LL(x,\Theta_2(x')),
\]
or equivalently,
\[
Fx=\sum_{x'\in ONB(Conf(\CC)_{m,n})}LL(x,\Theta_2(x'))x'.
\]
\end{Thm}

\begin{proof}
 We begin by considering the graphical interpretation of the product on the right-hand side, as illustrated in Figure \ref{fig:Fourier_pairing_proof1}, with a factor of:
 \[\mu^{-\#\text{dotted red loops}}\delta^{(m-1)(1-n)}\prod_{1\leq i\leq m-2,j}\sqrt{d_{Y_{i,j}}}\prod_{ i,j}\sqrt{d_{X_{i,j}}}.
 \] 
 By performing graphic calculus in the shaded region,
 \begin{figure}
     \centering
     \includegraphics[width=0.8\linewidth]{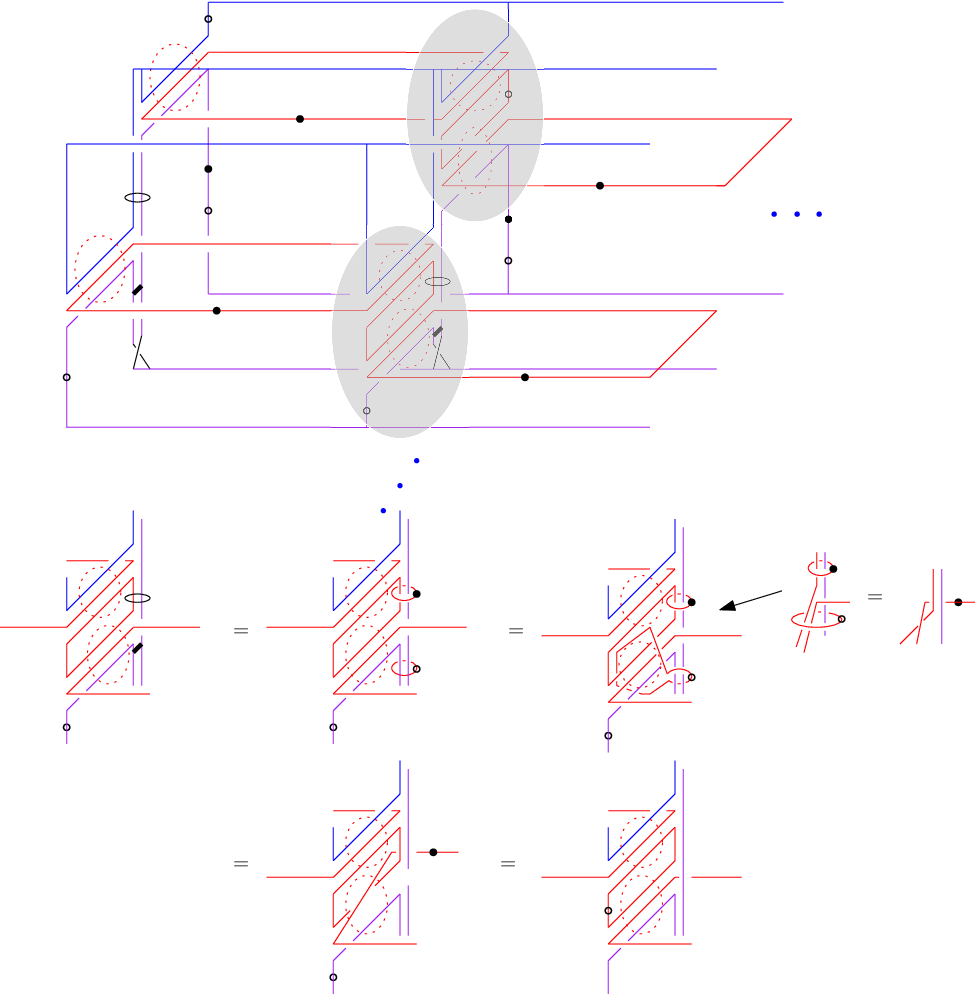}
     \caption{Graphic description of the Fourier transform} \label{fig:Fourier_pairing_proof1}
 \end{figure}
we obtain a more unified and streamlined diagram, shown in the first graph of Figure \ref{fig:Fourier_pairing_proof2}. 
\begin{figure}
     \centering
     \includegraphics[width=0.7\linewidth]{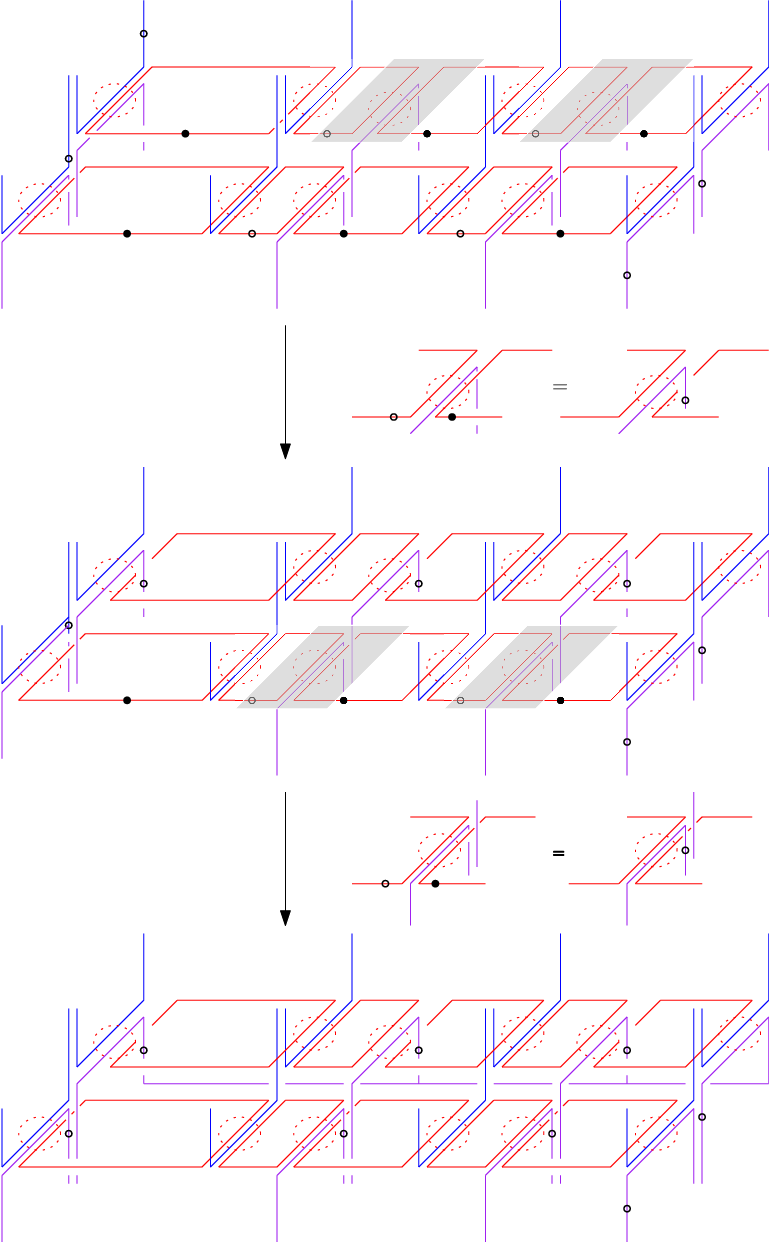}
     \caption{Graphic Calculus}     \label{fig:Fourier_pairing_proof2}
 \end{figure}
 
Then we continue the graphical calculus within the shaded region until we reach the final configuration. Applying a handle slide of the purple strands over the dotted red loops, followed by the resolution of all red loops, yields the configuration shown in Figure  \ref{fig:Fourier_pairing_proof3}, accompanied by a factor of $\delta^{(m-1)(1-n)}\prod_{1\leq i\leq m-2,j}\sqrt{d_{Y_{i,j}}}\prod_{ i,j}\sqrt{d_{X_{i,j}}}$. Finally, by bending the graph in the indicated direction and resolving all twists, we show that the resulting diagram evaluates to $LL(a,\Theta_2(a'))$.
\begin{figure}
    \centering
    \includegraphics[width=0.8\linewidth]{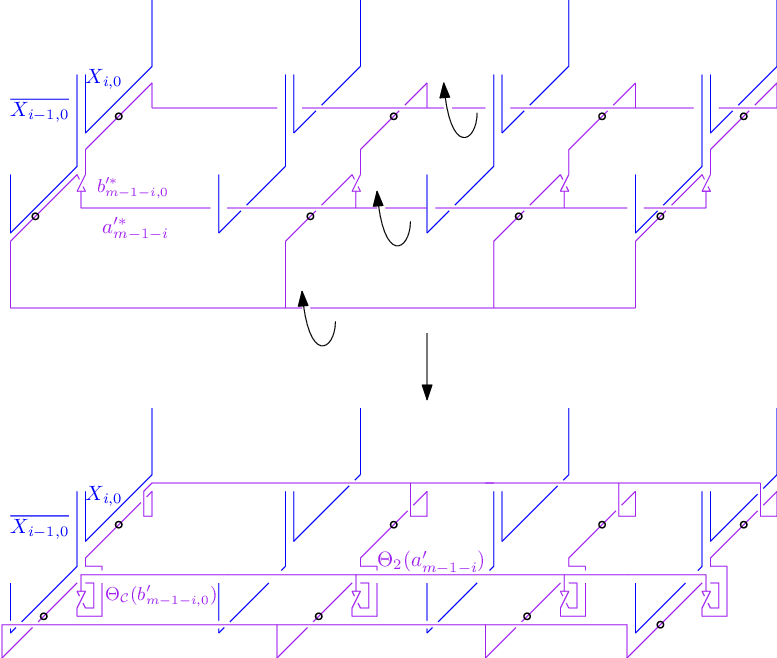}
    \caption{Graphic Calculus}
    \label{fig:Fourier_pairing_proof3}
\end{figure}
\end{proof}

From Theorem \ref{thm:Fourier_transform}, it follows immediately that the Fourier transform is a unitary operator. Consequently, we have
\begin{Cor}[\cite{LX19}]
 The Jones-Wassermann subfactors are selfdual.   
\end{Cor}

\begin{prop} \label{prop:FT_rels}
The following equalities hold:
\begin{align}
FT_{2i+1,j}=\eta T_{2m-2i-4,j}F\  (0\leq i\leq m-2),\ 
&FT_{2m-1,j}=\eta T_{2m-2,j}F\ (i=m-1),\\ 
FT_{2i,j+1}=\eta^{-1}T_{2m-2i-3,j}F\  (0\leq i\leq m-2),\ 
&FT_{2m-2,j+1}=\eta^{-1}T_{2m-1,j}F\ (i=m-1),\\
FT_{j+1}=\eta^{-m}T'_jF,\ &FT'_j=\eta^m T_{j}F.
\end{align}
\end{prop}
\begin{proof}
Using the graphic interpretation of the Fourier pairing. we have:
\[
\begin{aligned}
 LL(T_{2i+1,j}(x),\Theta_2(x'))&=\eta LL(x,T_{2i+2,j}\Theta_2(x')),\ LL(T_{2m-1,j}(x),\Theta_2(x'))=\eta LL(x,T_{0,j}\Theta_2(x')),\\
  LL(T_{2i,j+1}(x),\Theta_2(x'))&=\eta^{-1}LL(x,T_{2i+1,j}\Theta_2(x')),\ LL(T_{2m-2,j+1}(x),\Theta_2(x'))=\eta^{-1}LL(x,T_{2m-1,j}\Theta_2(x')).
 \end{aligned}
 \]
Now by Lemma \ref{lem:ThetaT_rel}, we have: 
\[
\begin{aligned}
FT_{2i+1,j}(x)=&\sum_{x'\in ONB}LL(T_{2i+1,j}(x),\Theta_2(x'))x',\\
=&\eta\sum_{x'\in ONB}LL(x,T_{2i+2,j}\Theta_2(x'))x',\\
=&\eta\sum_{x'\in ONB}LL(x,\Theta_2T^{-1}_{2m-2i-4,j}(x'))x',\qquad (\text{Lem.~\ref{lem:ThetaT_rel}})\\
=&\eta\sum_{x'\in ONB}LL(x,\Theta_2T^{-1}_{2m-2i-4,j}(x'))T_{2m-2i-4,j}T^{-1}_{2m-2i-4,j}x',\\
=&\eta T_{2m-2i-4,j}F(x).   
\end{aligned}
\]
The other identities can be derived similarly.
\end{proof}   

Since we have $FT_{2m-1,j}F^{-1}=\eta T_{2m-2,j}$ which implies: \[T_j\AA^{-1}_jT_{2m-1,j}\AA_jT_j^{-1}=\eta(\prod ^{n-2}_{k=j+1}\AA_kT_k^{-1}) uT_{2m-2,j}u^{-1}(\prod ^{j+1}_{k=n-2}T_k\AA^{-1}_k)=\eta T_{0,j},\] therefore we have:
\begin{equation}\label{eq:TA_rels_1}
\begin{aligned}
T_{2m-1,j}&=\eta \AA_jT_{0,j}\AA^{-1}_j,\\
&=\eta\prod^{0}_{i=2m-2}T_{i,j}T_{0,j}\prod^{2m-2}_{i=0}T^{-1}_{i,j},\\
&=\eta(\prod^{2}_{i=2m-2}T_{i,j})T_{1,j}T_{0,j}T^{-1}_{1,j}\prod^{2m-2}_{i=2}T^{-1}_{i,j},\\
&=\eta\eta^{-1}(\prod^{2}_{i=2m-2}T_{i,j})T^{-1}_{0,j}T_{1,j}T_{0,j}\prod^{2m-2}_{i=2}T^{-1}_{i,j},\qquad (\text{Prop.~\ref{prop:T_rels}})\\
&=T^{-1}_{0,j}(\prod^{2}_{i=2m-2}T_{i,j})T_{1,j}\prod^{2m-2}_{i=2}T^{-1}_{i,j}T_{0,j},\\
&=\cdots,\\
&=\eta \AA^{-1}_jT_{2m-2,j}\AA_j.     
\end{aligned}
\end{equation}
Similarly, we have $T_{2i+1,j}=\eta \AA_jT_{2i+2,j}\AA^{-1}_j=\eta \AA^{-1}_jT_{2m-2i-4,j}\AA_j$. Therefore we have:
\begin{equation}
\eta^{m} \AA_jT_j\AA^{-1}_j=\eta^{m} \AA^{-1}_jT_j\AA_j=T'_j.    
\end{equation}

\begin{lem}
The following identities hold:
\begin{align}
T_{2m-1,j}&=\eta \BB_jT_{2m-2,j+1}\BB^{-1}_j=\eta \BB^{-1}_jT_{0,j+1}\BB_j,\label{eq:TB_rels}\\
T_{2i+1,j}&=\eta\BB_jT_{2m-2i-4,j+1}\BB^{-1}_j=\eta \BB^{-1}_jT_{2i+2,j+1}\BB_j.
\end{align}
\end{lem}

\begin{proof}
From identity (\ref{eq:TA_rels_1}), it suffices to prove $\eta \AA_jT_{0,j}\AA^{-1}_j=\eta \BB^{-1}_jT_{0,j+1}\BB_j$. This is equivalent to $\BB_j\AA_jT_{0,j}=T_{0,j+1}\BB_j\AA_j$, which is identity (\ref{eq:BAT_rels_1}). The remaining cases follow by a similar argument.   
\end{proof}

Moreover, since $FT_{2i,j+1}F^{-1}=\eta^{-1}T_{2m-2i-3,j}$, and by Lemma \ref{lem:uT_rels}, $[uT_{2m-2i-3,j}u^{-1},T_{l,k}]=0$ for $k\geq j+2$, we obtain the following: 
\[
\begin{aligned}
T_{j+1}\AA^{-1}_{j+1}T_{j}\AA^{-1}_{j}T_{2i,j+1}\AA_{j}T_{j}^{-1}\AA_{j+1}T_{j+1}^{-1}=&\eta^{-1}(\prod ^{n-2}_{k=j+2}\AA_kT_k^{-1}) uT_{2m-2i-3,j}u^{-1}(\prod ^{j+2}_{k=n-2}T_k\AA^{-1}_k),\\
=&\eta^{-1} uT_{2m-2i-3,j}u^{-1}.    
\end{aligned}
\]
From the above computations, we have the following proposition.
\begin{prop}\label{prop:TAB_rels}
The following identities hold:
\[
\begin{aligned}
T_j\AA_j=\eta^{-m}\AA_jT'_j&,\ \  \AA_jT_j=\eta^{-m}T'_j\AA_j,
\\
T_{j+1}\BB_j=\eta^{-m}\BB_jT'_j&,\ \  \BB_jT_{j+1}=\eta^{-m}T'_j\BB_j,
\\
T_{j+1}\AA^{-1}_{j+1}T_{j}\AA^{-1}_{j}T_{j+1}=
&\eta^{-m} uT'_{j}u^{-1}T_{j+1}\AA^{-1}_{j+1}T_{j}\AA^{-1}_{j}.      
\end{aligned}
\]
\end{prop}
\begin{prop}   
\label{prop:FAB_rels}
We have the following identities (for $0\leq j\leq n-2$):
\begin{align}
 F\BB_j&=\AA_j F, \\
 F\AA_{j+1}&=\BB_jF.
\end{align}
\end{prop}
\begin{proof}
From Proposition \ref{prop:FT_rels} and identities \eqref{eq:TA_rels_1}, \eqref{eq:TB_rels}, we have:
 \[
 \begin{aligned}
  F\BB_{j}F^{-1}&=F(T_{0,j+1}\prod^{m-1}_{i=1}T_{2i-1,j}T_{2i,j+1})F^{-1},\\
  &=\eta^{-1}T_{2m-3,j}(\prod^{m-2}_{i=1}\eta T_{2m-2i-2,j}\eta^{-1}T_{2m-2i-3,j})\eta T_0\eta^{-1} T_{2m-1},\\
  &=\eta^{-1}T^{-1}_{2m-2}\AA_jT_{2m-1},\\
  &=\AA_j.\qquad \eqref{eq:TA_rels_1}\\
 F\AA_{j+1}F^{-1}&=F\prod_{i=0}^{2m-2}T_{2m-2-i,j+1}  F^{-1},\\
  &=\eta^{-1}T_{2m-1,j}\prod^{m-2}_{i=0}\eta T_{2i,j+1}\eta^{-1}T_{2i+1,j},\\
  &=\eta^{-1}T_{2m-1,j}\BB_jT^{-1}_{2m-2,j+1},\\
  &=\BB_jT_{2m-2,j+1}\BB^{-1}_j\BB_jT^{-1}_{2m-2,j+1},\qquad\eqref{eq:TB_rels}\\
  &=\BB_j.
 \end{aligned}
 \]
\end{proof}
Moreover, by conjugating $F$ to the identity (\ref{eq:BAT_rels_2}). we have 
\begin{equation}\label{eq:AB_rels}
\BB^{-1}_{j}\AA^{-1}_{j+1} T'_{j+1}=T'_{j+1}\AA^{-1}_{j+1}\BB^{-1}_{j}.  
\end{equation}
\begin{Thm}\label{Thm:F'F_rels}
Define the following unitary operator:
\[
 F'=u\prod^{n-2}_{j=0} T^{-1}_{n-2-j}\AA_{n-2-j}.
\]
We have the following identities:
\begin{align}
F^2=&F'^2=\rho_1,\\
F^{-1}F'=&\rho_2,\\
F\rho_2=&\rho^{-1}_2F.
\end{align}
\end{Thm}
\begin{proof}
From Lemma \ref{lem:utAB_rels}, Corollary \ref{Cor:ABTrho_1_rels} and identity \eqref{eq:AB_rels}, we have:
\[
\begin{aligned}
    F^2&=u^{-1}\prod^{n-2}_{j=0} T_{n-2-j}\AA^{-1}_{n-2-j}u^{-1}\prod^{n-2}_{j=0} T_{n-2-j}\AA^{-1}_{n-2-j},\\
    &=\prod^{n-2}_{j=0}\BB^{-1}_{n-2-j}\prod^{n-2}_{j=0} T_{n-2-j}\AA^{-1}_{n-2-j},\qquad (\text{Lem.~\ref{lem:utAB_rels}})\\
    &=\BB^{-1}_{n-2}(\prod^{n-1}_{j=0}\BB^{-1}_{n-3-j} T_{n-2-j}\AA^{-1}_{n-2-j})\AA^{-1}_0,\\
    &=\BB^{-1}_{n-2}(\prod^{n-1}_{j=0}\BB^{-1}_{n-3-j}\AA^{-1}_{n-2-j} T'_{n-2-j})\AA^{-1}_0,\\
    &=\BB^{-1}_{n-2}(\prod^{n-1}_{j=0}\AA^{-1}_{n-2-j}T_{n-2-j}\BB^{-1}_{n-3-j})\AA^{-1}_0,\qquad \eqref{eq:AB_rels}\\
    &=\rho_1. \qquad (\text{Cor.~\ref{Cor:ABTrho_1_rels}})
\end{aligned}
\]
The proof is similar for $F'^2=\rho_1$.

The proof of the second identity follows from a graphical calculus analogous to that used for $F$, we first establish the geometric interpretation of the pairing determined by $F'$, as illustrated in Figures \ref{fig:F'pairing1} and \ref{fig:F'pairing2}. It follows that, for any $x'\in Conf(\CC)_{m,n}$,
\[
<F'(x),x'>=LL(\rho_2(x),x')=<F\rho_2(x),x'>,
\]
which implies $F'=F\rho_2$.

For the last equality, since  $LL(\rho_2(x),\rho_2(x'))=LL(x,x')$ and $\rho_2\theta_2=\theta_2\rho^{-1}_2$, we have
\[
\begin{aligned}
F\rho_2(x)=&\sum_{x'\in ONB}LL(\rho_2(x),\Theta_2(x'))x',\\
=&\sum_{x'\in ONB}LL(x,\rho^{-1}_2\Theta_2(x'))x',\\
=&\sum_{x'\in ONB}LL(x,\Theta_2\rho_2(x'))x',\\
=&\sum_{x'\in ONB}LL(x,\Theta_2\rho_2(x'))\rho^{-1}_2\rho_2x',\\
=&\rho^{-1}_2F(x).   
\end{aligned}
\]
\begin{figure}
    \centering
    \includegraphics[width=0.8\linewidth]{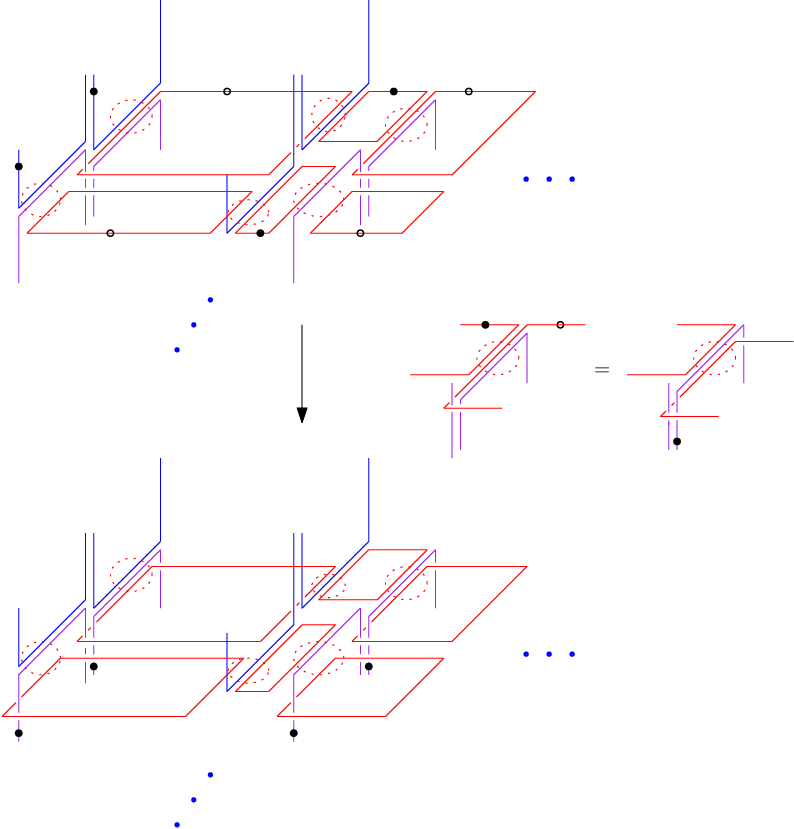}
    \caption{Graphic Calculus}
    \label{fig:F'pairing1}
\end{figure}

\begin{figure}
    \centering
    \includegraphics[width=0.8\linewidth]{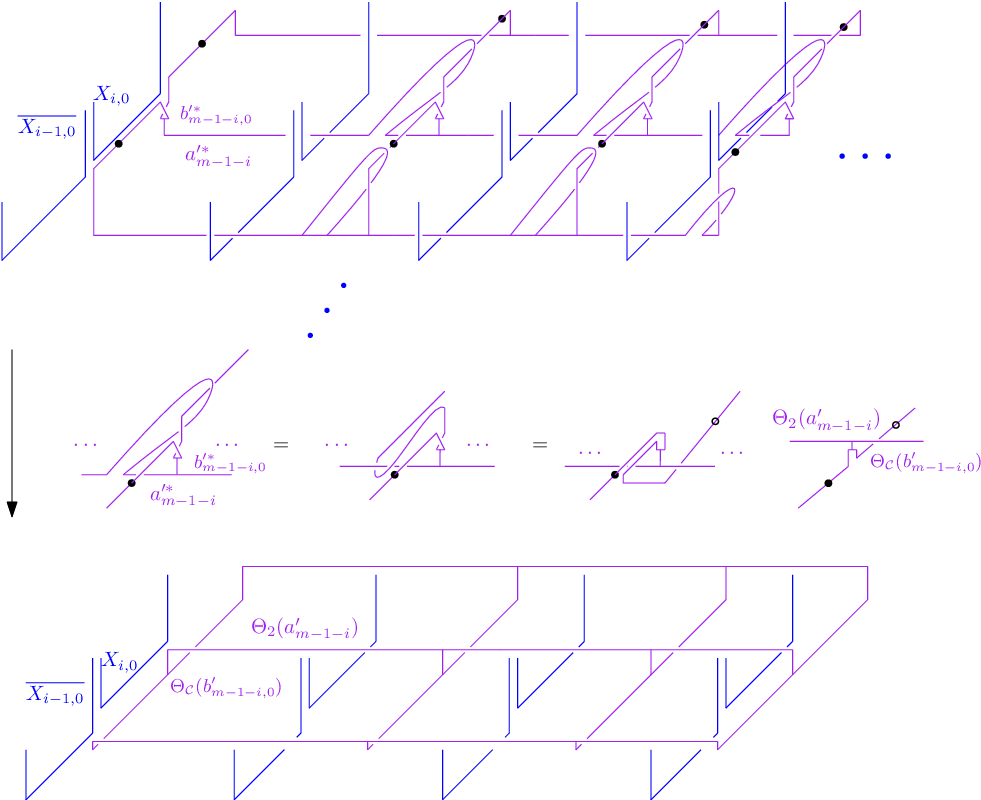}
    \caption{Graphic Calculus}
    \label{fig:F'pairing2}
\end{figure}
\end{proof}

Now, similar to Proposition \ref{prop:TAB_rels}, we have 
\begin{prop}\label{prop:F'TA_rels}
The following identities hold,
\[
\begin{aligned}
F'^{-1}T_{j}F'=\eta^{-m}T'_j,&\  F'^{-1}T'_{j}F'=\eta^{m}T_{j+1},\\
\AA^{-1}_j\AA^{-1}_{j+1}uT'_ju^{-1}&=\eta^{m}T_{j+1}\AA^{-1}_j\AA^{-1}_{j+1},\\
F'^{-1}\BB_jF'=\AA_{j+1},&\ F'^{-1}\AA_jF'=\BB_j. 
\end{aligned}
\]
\end{prop}
\begin{proof}
The proofs for the first two identities follow a similar argument to that provided in Proposition \ref{prop:FT_rels}. The third identity is derived analogously to the final identity in Proposition \ref{prop:TAB_rels}. Specifically, we first establish that:
\[
\AA^{-1}_jT_j\AA^{-1}_{j+1}T_{j+1}u^{-1}T'_ju(\AA^{-1}_jT_j\AA^{-1}_{j+1}T_{j+1})^{-1}=\eta^{m}T_{j+1},
\]
then from Lemma \ref{lem:uT_rels}, we have $T_jT_{j+1}u^{-1}T'_ju(T_jT_{j+1})^{-1}=uT'_ju^{-1}$. Combining these two identities yields the desired identity.
Finally, the proof for the last two identities follows from the arguments presented in Proposition \ref{prop:FAB_rels}.
\end{proof}

\begin{Cor}\label{Cor:rho_2TABu_rels}
The following relations hold:
\[
\begin{aligned}
[\rho_2,T_j]&=[\rho_2,T'_j]=0,\\
[\rho_2,\AA_j]&=[\rho_2,\BB_j]=0,\\
\rho_2 u&=u\rho^{-1}_2.
\end{aligned}
\]
\end{Cor}

\begin{proof}
By Theorem \ref{Thm:F'F_rels}, we have $F'=F\rho_2=\rho_2^{-1}F$, therefore $FF'^{-1}=F^{-1}F'=\rho$. Now, from Proposition \ref{prop:FT_rels}, \ref{prop:FAB_rels} and \ref{prop:F'TA_rels}, we have:
\[
\begin{aligned}
 \rho_2T_j&=FF'^{-1}T_j=\eta^{-m}FT'_jF'^{-1}=T_jFF'^{-1}=T_j\rho_2,\\
 \rho_2T'_j &=F^{-1}F'T'_j=\eta^{m}F^{-1}T_jF'=T'_jF^{-1}F'=T_j'\rho_2,\\     
 \rho_2\AA_j&=FF'^{-1}\AA_j=F\BB_jF'^{-1}=\AA_jFF'^{-1}=\AA_j\rho_2,\\
\rho_2\BB_j&=F^{-1}F'\BB_j=F^{-1}\AA_jF'=\BB_jF^{-1}F'=\BB_j\rho_2.
\end{aligned} 
\]
Therefore, from above calculations and Theorem \ref{Thm:F'F_rels}, we have:
\[
\begin{aligned}
\rho_2u&=\rho_2(\prod^{n-2}_{j=0} T_{n-2-j}\AA^{-1}_{n-2-j})F^{-1}\\
&=(\prod^{n-2}_{j=0} T_{n-2-j}\AA^{-1}_{n-2-j})\rho_2F^{-1}\\
&=(\prod^{n-2}_{j=0} T_{n-2-j}\AA^{-1}_{n-2-j})F^{-1}\rho^{-1}_2\\
&=u\rho_2^{-1}
\end{aligned}
\]
\end{proof}

\begin{Cor}\label{Cor:ABbraid_rels}
 We have the following identities:
 \[
 \begin{aligned} T_j\AA^{-1}_jT_{j+1}\AA^{-1}_{j+1}T_j\AA^{-1}_{j}&=T_{j+1}\AA^{-1}_{j+1}T_{j}\AA^{-1}_{j}T_{j+1}\AA^{-1}_{j+1} \\
\BB^{-1}_jT_{j+1}\BB^{-1}_{j+1}T_{j+2}\BB^{-1}_jT_{j+1}&=\BB^{-1}_{j+1}T_{j+2}\BB^{-1}_jT_{j+1}\BB^{-1}_{j+1}T_{j+2}  
 \end{aligned}
 \]
\end{Cor}
\begin{proof}
From Proposition \ref{prop:FAB_rels}, we have: 
\[
\begin{aligned}
&FT_{j+1}\AA^{-1}_{j+1}F^{-1}=\eta^{-m}T'_j\BB^{-1}_j\\
\implies &T_{j+1}\AA^{-1}_{j+1}T_{j}\AA^{-1}_{j}T_{j+1}\AA^{-1}_{j+1}(T_{j+1}\AA^{-1}_{j+1}T_{j}\AA^{-1}_{j})^{-1}=\ u\BB^{-1}_jT_{j+1}u^{-1}= T_j\AA^{-1}_{j}
\end{aligned}
\]
\end{proof}

\section{Braiding structures and their compatibility}
In this section, we demonstrate that the operators and relations established previously endow the multi-interval Jones–Wassermann subfactor planar algebra with specific braiding structures. We further show that this construction yields a projective unitary representation of the balanced superelliptic mapping class group $\SMod(\Sigma_{(n-1)(m-1)})$ (Theorem \ref{thm:Smod_rep}). 

To begin, we recall from \cite{LX19} that the multi-interval Jones–Wassermann subfactor planar algebra is characterized as an unshaded subfactor planar algebra, with the $n$-box spaces defined by the configuration spaces $\mathcal{P}_{n}=Conf(\CC)_{m,n}$.

\begin{defn}[\cite{GW17lifting}]
Let \(\sigma\in S_{2\ell}\). We say that \(\sigma\) is
\emph{parity-preserving} if
\[
    \sigma(i)\equiv i \pmod 2
    \quad \text{for all } i,
\]
and \emph{parity-reversing} if
\[
    \sigma(i)\equiv i+1 \pmod 2
    \quad \text{for all } i.
\]
A permutation is called \emph{parity-compatible} if it is either
parity-preserving or parity-reversing. The corresponding subgroup is
denoted by
\[
    W_{2\ell}\leq S_{2\ell}.
\]
\end{defn}

\begin{defn}
Let \(p:B_{2\ell}\to S_{2\ell}\) be the natural projection. A braid
\(\beta\in B_{2\ell}\) is called \emph{parity-compatible} if
\(\rho(\beta)\in W_{2\ell}\).
The corresponding subgroup of \(B_{2\ell}\) is denoted by
\[
    B_{2\ell}^{\mathrm{par}}:=\rho^{-1}(W_{2\ell})\leq B_{2\ell}.
\]    
\end{defn}

\begin{defn}
We define the generating tangles as illustrated in Figure \ref{fig:generators}.  Note that all of these generating braiding tangles are parity-compatible.   
\end{defn}

\begin{figure}
    \centering
\includegraphics[width=0.8\linewidth]{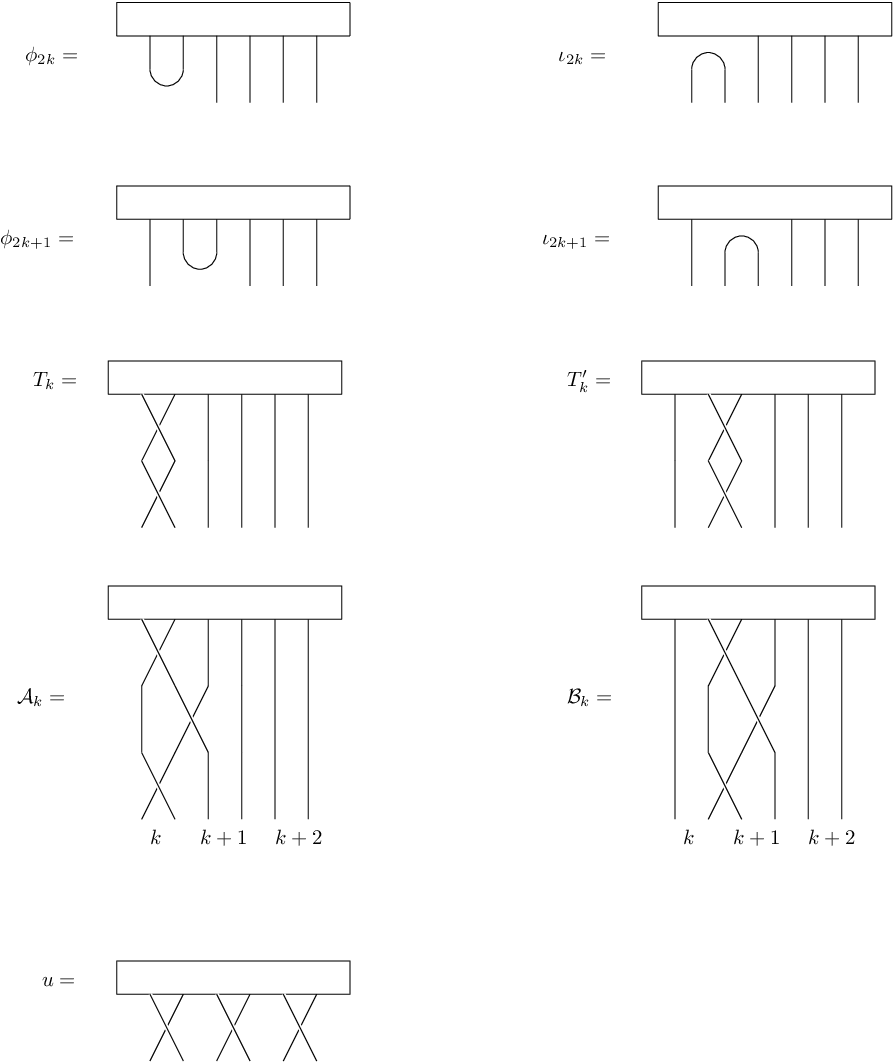}
    \caption{Generating tangles}
    \label{fig:generators}
\end{figure}

In the remainder of this section, we show that these generating braiding
tangles define a well-defined action of the parity-compatible braid
subgroup \(B_{2\ell}^{\mathrm{par}}\) on the planar algebra. We refer to
this as the parity-compatible braiding structure. Equivalently, isotopic
parity-compatible braids are assigned the same linear map. Moreover, this
action does more than satisfy the braid-group relations: it induces a projective
representation of the balanced superelliptic mapping class group.

\begin{rmk}
 The Zigzag relations for the contraction $\phi_k$ and inclusion $\iota_k$ follow directly from the definition via straightforward graphical calculus. In this framework, the relations are expressed as follows: 
 \begin{align*}
&\phi_{2k-1}\circ(\Id\otimes \iota_{2k})=(\Id\otimes \phi_{2k})\circ \iota_{2k-1}=\Id\\
&\phi_{2k-1}\circ(\iota_{2k-2}\otimes \Id )=(\phi_{2k-2}\otimes \Id)\circ \iota_{2k-1}=\Id
\end{align*}
\end{rmk}

\begin{defn}
To simplify our computations and avoid the complexities of tracking the scalar $\eta$ throughout our calculations, we adopt the following abuse of notation and normalization:

\begin{align*}
 &T_k:=\eta^{\frac{m}{2}} T_k,\ \ T_k':=\eta^{-\frac{m}{2}} T'_k,\ \ u:=\eta^{\frac{mn}{4}}u,\\
 &\AA_{k}:=\eta^{\frac{m}{4}}\AA_{k},\ \ \BB_{k}:=\eta^{\frac{m}{4}}\BB_{k}.
\end{align*}    
\end{defn}

Then we will have

\[
\AA^{\pm1}_iT_i=T_i'\AA_i^{\pm1},\quad \BB^{\pm1}_iT_{i+1}=T_i'\BB_i^{\pm1},
\]
and the only relations where $\eta$ appears will be the following (since $F_{new}=\eta^{-\frac{m}{4}}F$, $F'_{new}=\eta^{\frac{m}{4}}F'$):
\begin{align*}
 &F^2=\eta^{-\frac{m}{2}}\rho_1,\ \ (F')^2=\eta^{\frac{m}{2}}\rho_1, \\
 &\rho_2=\eta^{\frac{m}{2}}F^{-1}F'=\eta^{-\frac{m}{2}}FF'^{-1}.
\end{align*}

\subsection{Relations for the generating braiding tangles}

\begin{lem}\label{lem:Abraid_rels}
We have the following identities:
\begin{align*}
\AA^{-1}_jT_{j}\AA^{-1}_{j+1}\AA^{-1}_{j}=&\AA^{-1}_{j+1}\AA^{-1}_{j}\AA^{-1}_{j+1}T_{j+1},\\
\AA^{-1}_i\BB^{-1}_i\AA^{-1}_{i+1}\AA^{-1}_i=&\AA^{-1}_{i+1}\AA^{-1}_{i}\BB^{-1}_i\AA^{-1}_{i+1}.
\end{align*}
\end{lem}
\begin{proof}
Use the identity in Corollary \ref{Cor:ABbraid_rels}, we have: 
\begin{align*}
&T_j\AA^{-1}_jT_{j+1}\AA^{-1}_{j+1}T_j\AA^{-1}_{j}=T_{j+1}\AA^{-1}_{j+1}T_{j}\AA^{-1}_{j}T_{j+1}\AA^{-1}_{j+1},\\
\Leftrightarrow \quad&T_j\AA^{-1}_j\AA^{-1}_{j+1}T'_{j+1}T_j\AA^{-1}_{j}=T_{j}\AA^{-1}_{j+1}\AA^{-1}_{j}T'_{j+1}T_{j+1}\AA^{-1}_{j+1},\\
\Leftrightarrow
\quad&T_j\AA^{-1}_j\AA^{-1}_{j+1}\AA^{-1}_{j}T'_{j+1}T'_j=T_{j}\AA^{-1}_{j+1}\AA^{-1}_{j}\AA^{-1}_{j+1}T_{j+1}T'_{j+1},\\
\Leftrightarrow
\quad&\AA^{-1}_j\AA^{-1}_{j+1}\AA^{-1}_{j}T'_j=\eta^{m}\AA^{-1}_{j+1}\AA^{-1}_{j}\AA^{-1}_{j+1}T_{j+1},\\
\Leftrightarrow
\quad&\AA^{-1}_jT_{j}\AA^{-1}_{j+1}\AA^{-1}_{j}=\AA^{-1}_{j+1}\AA^{-1}_{j}\AA^{-1}_{j+1}T_{j+1}.
\end{align*}   
\end{proof}

For the second identity, according to Proposition \ref{prop:FAB_rels}, we have: 
\[
\begin{aligned}
&F\BB^{-1}_i\AA^{-1}_{i+1}F^{-1}=\AA^{-1}_i\BB^{-1}_i,\\
\implies 
\quad&T_{i+1}\AA^{-1}_{i+1}T_{i}\AA^{-1}_{i}\BB^{-1}_i\AA^{-1}_{i+1}(T_{i+1}\AA^{-1}_{i+1}T_{i}\AA^{-1}_{i})^{-1}=\AA^{-1}_i\BB^{-1}_i,\\
\implies
\quad&\AA^{-1}_{i+1}\AA^{-1}_{i}\BB^{-1}_i\AA^{-1}_{i+1}(\AA^{-1}_{i+1}\AA^{-1}_{i})^{-1}=(T_{i+1}T_{i})^{-1}\AA^{-1}_i\BB^{-1}_iT_{i+1}T_{i}=\AA^{-1}_i\BB^{-1}_i,\\
\implies
\quad&\AA^{-1}_i\BB^{-1}_i\AA^{-1}_{i+1}\AA^{-1}_i=\AA^{-1}_{i+1}\AA^{-1}_{i}\BB^{-1}_i\AA^{-1}_{i+1}. 
\end{aligned}
\]

\begin{defn}\label{def:t}
We define the following unitary operators for $(0\leq i<j\leq n-1)$, and $t_{i,i}=\Id$.
\begin{align*}
t_{2i+1,2j+1}:=
&\prod^{j-1}_{k=i}\AA^{-1}_k\prod^{i}_{k=j-1}\AA^{-1}_kT'_k,
&t'_{2i+1,2j+1}=&\prod^{i}_{k=j-1}T_k\AA^{-1}_k\prod^{j-1}_{k=i}\AA^{-1}_k,\\  
t_{2i+1,2j}:=
&(\prod^{j-2}_{k=i}\AA^{-1}_k)T^{-1}_{j-1}\prod^{i}_{k=j-2}\AA^{-1}_kT'_k
,&t'_{2i+1,2j}=&(\prod^{i}_{k=j-2}T'_k\BB^{-1}_k)T^{-1}_{i}\prod^{j-2}_{k=i}\BB^{-1}_k,\\
t_{2i+2,2j+1}:=
&(\prod^{j-2}_{k=i}\BB^{-1}_k)(T'_{j})^{-1}\prod^{i}_{k=j-2}\BB^{-1}_kT_{k+1},
&t'_{2i+2,2j+1}:=&(\prod^{i+1}_{k=j-1}T_k\AA^{-1}_k)(T'_{i})^{-1}\prod^{j-1}_{k=i+1}\AA^{-1}_k,\\  
t_{2i+2,2j}:=
&\prod^{j-2}_{k=i}\BB^{-1}_k\prod^{i}_{k=j-2}\BB^{-1}_kT_{k+1},
&t'_{2i+2,2j}:=&\prod^{i}_{k=j-2}T'_k\BB^{-1}_k\prod^{j-2}_{k=i}\BB^{-1}_k.
\end{align*}

\end{defn}

\begin{figure}
    \centering
    \includegraphics[width=1.1\linewidth]{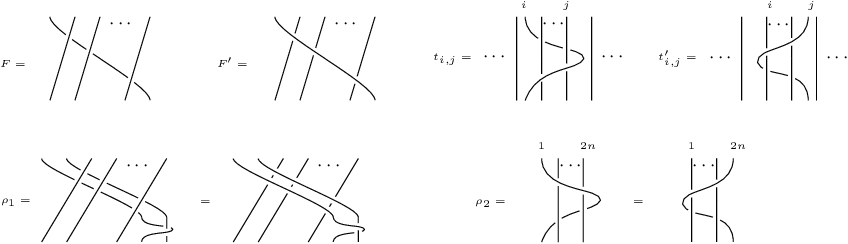}
    \caption{Some special tangles}
    \label{fig:special_tangles}
\end{figure}

It is straightforward that $Ft_{i+1,j+1}F^{-1}=t_{i,j}$ and $Ft'_{i+1,j+1}F^{-1}=t'_{i,j}$.
\begin{prop}\label{prop:t_comm_rels}
We have identities:
\begin{align}
[t_{i,j},t_{k,l}]&=0,\  for\ j<k,\  l<i\ or\  k<i<j\leq l,\\
[t'_{i,j},t'_{k,l}]&=0,\  for\ j<k,\  l<i\ or\  k\leq i<j<l,\\
[t_{i,j},t'_{k,l}]&=0,\  for\ j<k,\  l<i,\  k\leq i<j<l\ or\  i<k<l\leq j,\\
t_{i,j}t'_{i+1,j}&=t'_{i,j}t_{i,j-1}\label{eq:tt'_rels}.   
\end{align}
\end{prop}
\begin{proof}
We begin by proving the commutation relations for $t_{i,j}$. It suffices to prove the following identities, from which the relations for $t_{2i+2,2j+1}$ and $t_{2i+2,2j}$ follow by conjugation with $F$.  
\[
\begin{aligned}
[t_{2i+1,2j+1},\AA^{-1}_k]&=[t_{2i+1,2j+1},T_{k}]=0,\ for\ i+1\leq k\leq j-1,\\
[t_{2i+1,2j+1},\BB^{-1}_k]&=[t_{2i+1,2j+1},T'_l]=0,\ for\ i\leq k\leq j-2\ and\ i\leq l\leq j-1,\\
[t_{2i+1,2j+2},\AA^{-1}_k]&=[t_{2i+1,2j+2},T_{l}]=0,\ for\ i+1\leq k\leq j-1\ and\ i+1\leq l\leq j,\\
[t_{2i+1,2j+2},\BB^{-1}_k]&=[t_{2i+1,2j+2},T'_k]=0,\ for\ i\leq k\leq j-1.
\end{aligned}
\]

Now we have:
\[
\begin{aligned}
t_{2i+1,2j+1}\AA^{-1}_l&=(\prod^{j-1}_{k=i}\AA^{-1}_k\prod^{i}_{k=j-1}T_k\AA^{-1}_k)\AA^{-1}_l,\\
&=(\prod^{j-1}_{k=i}\AA^{-1}_k)T_{j-1}\AA^{-1}_{j-1}\cdots T_{l}\AA^{-1}_{l}T_{l-1}\AA^{-1}_{l-1}\AA^{-1}_l\prod^{i}_{k=l-2}T_k\AA^{-1}_k, \\
&=(\prod^{j-1}_{k=i}\AA^{-1}_k)T_{j-1}\AA^{-1}_{j-1}\cdots \underline{T_{l}\AA^{-1}_{l}T_{l-1}\AA^{-1}_{l-1}T_l\AA^{-1}_l}T'^{-1}_l\prod^{i}_{k=l-2}T_k\AA^{-1}_k,\qquad (\text{Lem.~\ref{lem:Abraid_rels}})\\
&=\AA^{-1}_i\cdots\underline{\AA^{-1}_{l-1}\AA^{-1}_{l}T_{l-1}\AA^{-1}_{l-1}T^{-1}_l}(\prod^{j-1}_{k=l+1}\AA^{-1}_k)\prod^{i}_{k=j-1}T_k\AA^{-1}_k,\\
&=\AA^{-1}_i\cdots \AA^{-1}_{l}\AA^{-1}_{l-1}\AA^{-1}_{l}(\prod^{j-1}_{k=l+1}\AA^{-1}_k)\prod^{i}_{k=j-1}T_k\AA^{-1}_k,\\
&=\AA^{-1}_lt_{2i+1,2j+1}.
\end{aligned}
\]
Next, from Proposition \ref{prop:TAB_rels} and \ref{prop:F'TA_rels}, we have:
\[
\begin{aligned}
t_{2i+1,2j+1}T_l&=(\prod^{j-1}_{k=i}\AA^{-1}_k\prod^{i}_{k=j-1}T_k\AA^{-1}_k)T_l,\\
&=(\prod^{j-1}_{k=i}\AA^{-1}_k)T_{j-1}\AA^{-1}_{j-1}\cdots \underline{T_{l}\AA^{-1}_{l}T_{l-1}\AA^{-1}_{l-1}T_l}\prod^{i}_{k=l-2}T_k\AA^{-1}_k, \qquad (\text{Prop.~\ref{prop:TAB_rels}})\\ 
&=\eta^{-m}(\prod^{j-1}_{k=i}\AA^{-1}_k)T_{j-1}\AA^{-1}_{j-1}\cdots uT'_{l-1}u^{-1}T_{l}\AA^{-1}_{l}T_{l-1}\AA^{-1}_{l-1}\prod^{i}_{k=l-2}T_k\AA^{-1}_k, \\
&=\eta^{-m}\AA^{-1}_i\cdots\underline{\AA^{-1}_{l-1}\AA^{-1}_{l}uT'_{l-1}u^{-1}}(\prod^{j-1}_{k=l+1}\AA^{-1}_k)\prod^{i}_{k=j-1}T_k\AA^{-1}_k, \qquad(\text{Prop.~\ref{prop:F'TA_rels}})\\
&=T_lt_{2i+1,2j+1}.
\end{aligned}
\]

From Proposition \ref{prop:ABC_rels} and equation (\ref{eq:AB_rels}), we have:
\[
\begin{aligned}
t_{2i+1,2j+1}\BB^{-1}_l&=(\prod^{j-1}_{k=i}\AA^{-1}_k\prod^{i}_{k=j-1}T_k\AA^{-1}_k)\BB^{-1}_l,\\
&=(\prod^{j-1}_{k=i}\AA^{-1}_k)T_{j-1}\AA^{-1}_{j-1}\cdots T_{l+1}\AA^{-1}_{l+1}\underline{T_{l}\AA^{-1}_{l}\BB^{-1}_l}\prod^{i}_{k=l-1}T_k\AA^{-1}_k, \\ 
&=(\prod^{j-1}_{k=i}\AA^{-1}_k)T_{j-1}\AA^{-1}_{j-1}\cdots T_{l+1}\underline{\AA^{-1}_{l+1}T_{l+1}\BB^{-1}_l}\AA^{-1}_{l}\prod^{i}_{k=l-1}T_k\AA^{-1}_k,\qquad (\text{Eq.~\eqref{eq:BAT_rels_2}}) \\ 
&=\eta^{-m}(\prod^{j-1}_{k=i}\AA^{-1}_k)T_{j-1}\AA^{-1}_{j-1}\cdots \BB^{-1}_lT'_{l}T_{l+1}\AA^{-1}_{l+1}\AA^{-1}_{l}\prod^{i}_{k=l-1}T_k\AA^{-1}_k, \qquad (\text{Eq.~\eqref{eq:AB_rels}}) \\
&=\eta^{-m}\AA^{-1}_i\cdots\AA^{-1}_{l}\underline{\AA^{-1}_{l+1}\BB^{-1}_lT'_l}T^{-1}_l(\prod^{j-1}_{k=l+2}\AA^{-1}_k)\prod^{i}_{k=j-1}T_k\AA^{-1}_k,\\
&=\eta^{-m}\AA^{-1}_i\cdots\underline{\AA^{-1}_{l}T'_l\BB^{-1}_l}\AA^{-1}_{l+1}T^{-1}_l(\prod^{j-1}_{k=l+2}\AA^{-1}_k)\prod^{i}_{k=j-1}T_k\AA^{-1}_k,\\
&=\AA^{-1}_i\cdots\BB^{-1}_l\AA^{-1}_{l}T_l\AA^{-1}_{l+1}T^{-1}_l(\prod^{j-1}_{k=l+2}\AA^{-1}_k)\prod^{i}_{k=j-1}T_k\AA^{-1}_k,\\
&=\BB^{-1}_lt_{2i+1,2j+1}.
\end{aligned}
\]

The commutation relation $[t_{2i+1,2j+1},T'_l]=0$ is straightforward and follows directly from Proposition \ref{prop:TAB_rels}.
Regarding the case of $t_{2i+1,2j+2}$, the non-trivial relations reduce to $[t_{2i+1,2j+2},\BB^{-1}_{j-1}]=[t_{2i+1,2j+2},T_j]=0$. These can be derived from the following two identities:
\[\AA^{-1}_{j-1}T^{-1}_jT_{j-1}\AA^{-1}_{j-1}\BB^{-1}_{j-1}=\AA^{-1}_{j-1}\BB^{-1}_{j-1}\AA^{-1}_{j-1}=\BB^{-1}_{j-1}\AA^{-1}_{j-1}T^{-1}_jT_{j-1}\AA^{-1}_{j-1}.\qquad (\text{Eq.~\eqref{eq:BAT_rels_2}})
\]
\[
\begin{aligned}
 \AA^{-1}_{j-1}T^{-1}_jT_{j-1}\AA^{-1}_{j-1}T_{j}&=uT^{-1}_{j-1}\BB^{-1}_{j-1}T_jT_j^{-1}\BB^{-1}_{j-1}T^2_ju^{-1},\qquad (\text{Lem.~\ref{lem:utAB_rels}})\\
 &=\eta^{-m}uT^{-1}_{j-1}\BB^{-1}_{j-1}T'_{j-1}\BB^{-1}_{j-1}T^2_ju^{-1},\\
 &=uT_{j}T^{-1}_{j-1}\BB^{-1}_{j-1}\BB^{-1}_{j-1}T^2_ju^{-1},\\
 &=T_{j}\AA^{-1}_{j-1}T^{-1}_jT_{j-1}\AA^{-1}_{j-1}.   
\end{aligned}
\]

The commutation relations for the operators $t'_{i,j}$ follow a similar derivation to those of $t_{i,j}$. The only distinct case is the relation $[t'_{2i+1,2j+1},T'_l]=0$ for $i\leq l\leq j-2$, which can be deduced from the following identity:
\[
\begin{aligned}
 T_{l+1}\AA^{-1}_{l+1}T_{l}\AA^{-1}_{l}\AA^{-1}_l\AA^{-1}_{l+1}T'_l &=F\BB^{-1}_{l+1}\BB^{-1}_{l}T_{l+2}T_{l+1}\BB^{-1}_l\BB^{-1}_{l+1}T_{l+1}F^{-1},\qquad (\text{Prop.~\ref{prop:FAB_rels}})\\
 &=FuT^{-1}_{l+2}\AA^{-1}_{l+1}T_{l+1}T^{-1}_{l+1}\AA^{-1}_lT_lT_{l+2}T_{l+1}T^{-1}_{l+1}\AA^{-1}_lT_lT^{-1}_{l+2}\AA^{-1}_{l+1}T^2_{l+1}u^{-1}F^{-1},\  (\text{Lem.~\ref{lem:utAB_rels}})\\
 &=FuT^{-1}_{l+2}\AA^{-1}_{l+1}\AA^{-1}_lT_l\AA^{-1}_lT_l\AA^{-1}_{l+1}T^2_{l+1}u^{-1}F^{-1},\\
 &=FuT_{l+1}T^{-1}_{l+2}\AA^{-1}_{l+1}\AA^{-1}_lT_l\AA^{-1}_lT_l\AA^{-1}_{l+1}T_{l+1}u^{-1}F^{-1},\\
 &=T'_lT_{l+1}\AA^{-1}_{l+1}T_{l}\AA^{-1}_{l}\AA^{-1}_l\AA^{-1}_{l+1}.
\end{aligned}
\]

Now, from the commutation relations obtained above, we have:
\begin{align*}
&t_{2i+1,2j+1}t'_{2i+2,2j+1}=t'_{2i+1,2j+1}t_{2i+1,2j},\\
\Leftrightarrow\quad& t'_{2i+2,2j+1}t_{2i+1,2j+1}=t'_{2i+1,2j+1}t_{2i+1,2j},\\
\Leftrightarrow\quad& (\prod^{i+1}_{k=j-1}T_k\AA^{-1}_k)(T'_{i})^{-1}\prod^{j-1}_{k=i+1}\AA^{-1}_k\prod^{j-1}_{k=i}\AA^{-1}_k\prod^{i}_{k=j-1}\AA^{-1}_kT'_k=\prod^{i}_{k=j-1}T_k\AA^{-1}_k\prod^{j-1}_{k=i}\AA^{-1}_k(\prod^{j-2}_{k=i}\AA^{-1}_k)T^{-1}_{j-1}\prod^{i}_{k=j-2}\AA^{-1}_kT'_k,\\
\Leftrightarrow\quad& (T'_{i})^{-1}\prod^{j-1}_{k=i+1}\AA^{-1}_k\prod^{j-1}_{k=i}\AA^{-1}_kT_{j-1}\AA^{-1}_{j-1}=T_i\AA^{-1}_i\prod^{j-1}_{k=i}\AA^{-1}_k(\prod^{j-2}_{k=i}\AA^{-1}_k)T^{-1}_{j-1}.
\end{align*}
Now we have: 
\begin{align*}
(T'_{i})^{-1}\prod^{j-1}_{k=i+1}\AA^{-1}_k\prod^{j-1}_{k=i}\AA^{-1}_kT_{j-1}\AA^{-1}_{j-1}&=(T'_{i})^{-1}(\prod^{j-2}_{k=i+1}\AA^{-1}_k\prod^{j-3}_{k=i}\AA^{-1}_k)\underline{\AA^{-1}_{j-1}\AA^{-1}_{j-2}\overline{(\AA}_{j-1}T_{j-1})^2}T^{-1}_{j-1},\\
&= (T'_{i})^{-1}(\prod^{j-3}_{k=i+1}\AA^{-1}_k\prod^{j-4}_{k=i}\AA^{-1}_k)\underline{\AA^{-1}_{j-2}\AA^{-1}_{j-3}\overline{(\AA}_{j-2}T_{j-2})^2}\AA^{-1}_{j-1}\AA^{-1}_{j-2}T^{-1}_{j-1},\  (\text{Lem.~\ref{lem:Abraid_rels}})\\
&= \cdots,\\
&= (T'_{i})^{-1}(\AA^{-1}_iT_i)^2\prod^{j-1}_{k=i+1}\AA^{-1}_k(\prod^{j-2}_{k=i}\AA^{-1}_k)T^{-1}_{j-1},\\
&=T_i\AA^{-1}_i\prod^{j-1}_{k=i}\AA^{-1}_k(\prod^{j-2}_{k=i}\AA^{-1}_k)T^{-1}_{j-1}.
\end{align*}

Moreover, we have:
\[
\begin{aligned}
 &t_{2i+1,2j}t'_{2i+2,2j}=t'_{2i+1,2j}t_{2i+1,2j-1},\\
\Leftrightarrow\quad& t'_{2i+2,2j}t_{2i+1,2j}=t'_{2i+1,2j}t_{2i+1,2j-1},\\
\Leftrightarrow\quad& (\prod^{j-2}_{k=i}\BB^{-1}_k)\prod^{j-2}_{k=i}\AA^{-1}_k T^{-1}_{j-1}=T_i^{-1}\prod^{j-2}_{k=i}\BB^{-1}_k\prod^{j-2}_{k=i}\AA^{-1}_k,\\
\end{aligned}
\]   
which follows from Proposition \ref{prop:ABC_rels}.
\end{proof}

\begin{defn}
 We define the following unitary operators:
\[
 r_{i,j}:=\prod^{j}_{k=i}t_{k,j}=\prod^{j}_{k=i}t'_{i,k}.
\]
\end{defn}

\begin{proof}
In view of Proposition \ref{prop:t_comm_rels}, the product is independent of the order of its factors. First, note that $r_{i,i+1}=t_{i,i+1}=t'_{i,i+1}$. Now  
assume the relations hold for $j-i < k$, when $j-i=k$, we have: 
\[
\begin{aligned}
\prod^{j+1}_{k=i}t_{k,j+1}&=t_{i,j+1}\prod^{j+1}_{k=i+1}t_{k,j+1},\\
&=t_{i,j+1}\prod^{j+1}_{k=i+1}t'_{i+1,k}, \qquad (\text{By induction})\\
&=t'_{i,j+1}t_{i,j}\prod^{j}_{k=i+1}t'_{i+1,k},\qquad (\text{Eq.~\eqref{eq:tt'_rels}})\\
&=\cdots,\\
&=\prod^{j+1}_{k=i}t'_{i,k}.
\end{aligned}
\]
\end{proof}

\begin{prop}\label{prop:r_comm_rels}
 The following relations hold in $\End(Conf(\CC)_{m,n})$:
\[
\begin{aligned}
&r_{1,2n}=\Id,\\
&[r_{i,j},r_{k,l}]=0,\  for\ j<k,\  l<i\ or\  k\leq i<j\leq l,\\
&[r_{i,j},\prod^{j}_{k=i} t_{k,l}]=[r_{i,j},\prod^{j}_{k=i} t'_{s,k}]=0,\  for\ j\leq l, s\leq i.
\end{aligned}
\]
\end{prop}
\begin{proof}
We first prove the following identity:
\[
t_{1,2k}t_{2,2k}=T_0^{-1}\prod^{k-2}_{l=0}C_{l,l+1},
\]
the identity holds when $k=1$, suppose it holds for $k$, then we have:
\[
\begin{aligned}
t_{1,2k+2}t_{2,2k+2}=&t_{1,2k+2}T_1\BB^{-1}_0t_{4,2k+2}\BB^{-1}_0,\\
=&T_1\BB^{-1}_0t_{1,2k+2}t_{4,2k+2}\BB^{-1}_0,\\
=&T_1\BB^{-1}_0\AA^{-1}_0t_{3,2k+2}t_{4,2k+2}\AA^{-1}_0T'_0\BB^{-1}_0,\\
=&C^{-1}_{0,1}F^{-2}t_{1,2k}t_{2,2k}F^2C^{-1}_{0,1},\\
=&C^{-1}_{0,1}T_1^{-1}\prod^{k-1}_{l=1}C^{-1}_{l,l+1}(\prod^{1}_{l=k-1}C^{-1}_{l,l+1})C^{-1}_{0,1},\\
=&T_0^{-1}\prod^{k-1}_{l=0}C^{-1}_{l,l+1}\prod^{0}_{l=k-1}C^{-1}_{l,l+1}.
\end{aligned}
\]
 Therefore, we have:
\[
\begin{aligned}
r_{1,2n}&=\prod^{2n}_{k=1}t_{2k-1,2n}t_{2k,2n},\\ 
&=\prod^{n-1}_{l=0}T_l^{-1}\prod^{n-2}_{s=0}(\prod^{n-2}_{l=s}C^{-1}_{l,l+1}\prod^{s}_{l=n-2}C^{-1}_{l,l+1}).
\end{aligned}
\]
Recall that the $C_{i,i+1}$ are braids in $\CC^{\boxtimes m}$ satisfying the standard braid relations. Consequently, the operator represented by the right-hand side corresponds to a full right twist applied to each $a_i\in \Hom_{\CC}(\mathbbm{1}, \bigotimes^{n-1}_{j=0} Y_{i,j})$ for $0\leq i\leq m-1$, which evaluates to the identity.
The commutation relation follows from the first identity and Proposition \ref{prop:t_comm_rels}.
\end{proof}

\begin{lem}\label{lem:rho_2tr_rels}
The following relations hold in $\End(Conf(\CC)_{m,n})$:
\begin{align*}
 &\eta^{\frac{m}{2}}t^{-1}_{1,2n}=\eta^{-\frac{m}{2}}t'_{1,2n}=\rho_2, &&\eta^{-\frac{m}{2}}r^{-1}_{1,2n-1}=\eta^{\frac{m}{2}}r_{2,2n}=\rho_2, \\
 &t^{-1}_{2j+1,2n}=\eta^{-\frac{m}{2}}t'_{1,2j+1}\rho_2,
 &&t^{-1}_{2j,2n}=\eta^{-\frac{m}{2}}t'_{1,2j}\rho^{-1}_2,\\
 &r_{1,2k}=\eta^{-m(n-k)}r_{2k+1,2n},
 &&r_{1,2k+1}=\eta^{-m(n-k-\frac{1}{2})}r_{2k+2,2n}\rho^{-1}_2.
\end{align*}
\end{lem}
\begin{proof}
The first identity follows from the definition of $F,F'$ and $\rho_2=F^{-1}F'=FF'^{-1}$. For the second identity, we have
\[
\begin{aligned}
 &\rho_2=\eta^{\frac{m}{2}}t^{-1}_{1,2n}=\eta^{\frac{m}{2}}(\prod^{j-1}_{k=0}\AA_kT^{-1}_k)t^{-1}_{2j+1,2n}\prod^{0}_{k=j-1}\AA_k,\\ 
 \implies&\ t^{-1}_{2j+1,2n}=\eta^{-\frac{m}{2}}(\prod^{0}_{k=j-1}\AA^{-1}_kT'_k)\rho_2\prod^{j-1}_{k=0}\AA^{-1}_k=\eta^{-\frac{m}{2}}\rho_2t'_{1,2j+1},\\
 &\rho_2=\eta^{-\frac{m}{2}}t'_{1,2n}=\eta^{-\frac{m}{2}}(\prod^{j-1}_{k=n-2}T'_k\BB^{-1}_k)t'_{1,2j}\prod^{n-2}_{k=j-1}\BB^{-1}_k,\\
 \implies&\ t'_{1,2j}=\eta^{\frac{m}{2}}(\prod^{n-2}_{k=j-1}T^{-1}_{k+1}\BB_k)\rho_2\prod^{j-1}_{k=n-2}\BB_k=\eta^{\frac{m}{2}}\rho_2t^{-1}_{2j,2n}.
\end{aligned}
\]

The third relation follows from Proposition \ref{prop:r_comm_rels}, 
\[
r_{1,2n-1}=r_{1,2n}t'^{-1}_{1,2n}=\Id\rho_2^{-1}=\rho_2^{-1}.
\]

The final two relations follow similarly from Proposition \ref{prop:r_comm_rels}, By applying the second identity, we obtain:
\[
\begin{aligned}
r_{1,2k}&=r_{1,2n}\prod^{2n}_{s=2k+1}t'^{-1}_{1,s},\\
&=\prod^{n-1}_{s=k}t'^{-1}_{1,2s+1}t'^{-1}_{1,2s+2},\\
&=\prod^{n-1}_{s=k}\eta^{-m}\rho_2t_{2s+1,2n}\rho^{-1}_2t_{2s+2,2n},\\
&=\eta^{-m(n-k)}\prod^{2n}_{s=2k+1}t_{s,2n}=\eta^{-m(n-k)}r_{2k+1,2n}.
\end{aligned}
\]

Similarly, we have: 
\[
\begin{aligned}
 r_{1,2k+1}&=\prod^{2n}_{s=2k+2}t'^{-1}_{1,s},\\
 &=\eta^{-m(n-k-1)}t'^{-1}_{1,2k+2}r_{2k+3,2n},\\
 &=\eta^{-m(n-k-\frac{1}{2})}\rho^{-1}_2r_{2k+2,2n}.
\end{aligned}
\]
\end{proof}

Now from the definition of $\rho_2$, we have:
\begin{equation}\label{r_powerthree_to_1}
\Id=\rho^m_2=\eta^{-\frac{3m}{2}}r^{-3}_{1,2n-1}\implies r^3_{1,2n-1}=\eta^{-\frac{3m}{2}}\Id.   
\end{equation}

Recall from Proposition \ref{prop:ABC_rels} that the standard braidings of $\CC^{\boxtimes m}$ can be represented as operators within $\End(Conf(\CC)_{m,n})$. Furthermore, using the planar tangle interpretation illustrated in Figure \ref{fig:connect_two_braidings}, we establish a direct correspondence between the braiding of double strings and the braiding of $\CC^{\boxtimes m}$ (cf. \cite[Prop.~4.17]{LR24}).

\begin{figure}
    \centering
    \includegraphics[width=0.6\linewidth]{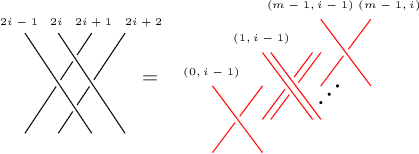}
    \caption{Braiding of the double strings equals the braiding of $\CC^{\boxtimes n}$}
    \label{fig:connect_two_braidings}
\end{figure}

\begin{defn}
Let $\Delta_n$ be the unitary action of half-full-twist in the braid group of $n$-strands on the Hom space of $\CC$, and it follows readily from the graphical calculus that $[\Delta_n^{\boxtimes m},u]=0$. We define the following unitary operator in $\End(Conf(\CC)_{m,n})$:
\[
\Delta:=\Delta_n^{\boxtimes m}u=u\Delta_n^{\boxtimes m}.
\]
\end{defn}

Given that $\Delta^2_n$ represents the full twist on $n$ strands, and utilizing the twist property, we can express this operator as the product of individual twists acting on each strand. This leads to the following identity: 
\begin{equation}\label{delta_square_to_1}
 \Delta^2=\Id.   
\end{equation}

\begin{lem}\label{lem:Delta_ABT_rels}
 We have the following relations:
 \[
 \begin{aligned}
  T_i\Delta=\Delta T_{n-1-i},&\quad  T'_i\Delta=\Delta T'_{n-2-i}, \\
  \BB_i\Delta=\Delta\AA_{n-2-i},&\quad \AA_i\Delta=\Delta\BB_{n-2-i}.\\
 \end{aligned}
 \]
\end{lem}

\begin{proof}
From the topological definition of $\Delta$, we have: 
\begin{align*}
    T_i'\Delta&=\Delta u^{-1}C^{-1}_{i,i+1}T'_{n-2-i}C_{i,i+1}u,\\
    \AA_i\Delta&=\Delta u^{-1}C^{-1}_{i,i+1}\AA_{n-2-i}C_{i,i+1}u,\\
    \BB_i\Delta&=\Delta u^{-1}C^{-1}_{i,i+1}\BB_{n-2-i}C_{i,i+1}u.
\end{align*}
The identities can now be deduced from the following identities:
 \begin{align*}
T'_iC_{i,i+1}u&=T'_i\AA_i\BB_iT^{-1}_{i+1}u,\\
&=T'_i\AA_iuT_i^{-1}\AA_i, \qquad (\text{Lem.~\ref{lem:utAB_rels}})\\
&=\AA_iuT_i^{-1}\AA_iT'_i, \qquad (\text{Prop.~\ref{prop:TAB_rels}})\\
&=C_{i,i+1}uT'_i. \qquad (\text{Lem.~\ref{lem:utAB_rels}})
 \end{align*}
 \[
 \begin{aligned}
\AA_iC_{i,i+1}u&=\AA_i\BB_i\AA_iT^{-1}_{i}u,\\
&=C_{i,i+1}T_{i+1}\AA_iT_i^{-1}u,\\
&=C_{i,i+1}u\BB_i. \qquad (\text{Lem.~\ref{lem:utAB_rels}})
 \end{aligned}
 \]
 \begin{align*}
\BB_iC_{i,i+1}u&=\BB_iT^{-1}_{i}\AA_i\BB_iu,\\
&=C_{i,i+1}T_{i}T_i'^{-1}\BB_iu,\\
&=C_{i,i+1}u\AA_i. \qquad (\text{Lem.~\ref{lem:utAB_rels}})    
 \end{align*}
\end{proof}
From the Lemma \ref{lem:Delta_ABT_rels} and the definition of $t_{i,j},t'_{i,j}$, we have the following: 
\[
\begin{aligned}
 t_{2i+1,2j+1}\Delta=\Delta t_{2n-2j,2n-2i}, &\quad t_{2i+1,2j}\Delta=\Delta t_{2n-2j+1,2n-2i},
  \\
 t_{2i+2,2j+1}\Delta=\Delta t_{2n-2j,2n-2i-1},&\quad t_{2i+2,2j}\Delta=\Delta t_{2n-2j+1,2n-2i-1}.
\end{aligned}
\]
\begin{prop}\label{prop:r_delta_rels}

We have the following relations:
\[
\begin{aligned}
r_{i,j}\Delta&=\Delta r_{2n+1-j,2n+1-i},\\ 
\rho_2\Delta&=\Delta\rho_2^{-1},\\
r_{1,2n-2i}\Delta&=\eta^{m(n-k)}\Delta r_{1,2i},\\
r_{1,2n-2i-1}\Delta&=\eta^{m(n-k-1)}\Delta r_{1,2i+1}r^{-1}_{1,2n-1}.
\end{aligned}
\]
\end{prop}
\begin{proof}
The first identity is straightforward, for the second identity, from Lemma \ref{lem:rho_2tr_rels}, we have: 
\[
\rho_2\Delta=\eta^{\frac{m}{2}}t^{-1}_{1,2n}\Delta=\Delta \eta^{\frac{m}{2}}t'^{-1}_{1,2n}=\Delta\rho_2^{-1}.
\]

Now, the last two identities also follow from Lemma \ref{lem:rho_2tr_rels}. 
\begin{align*}
&r_{1,2n-2i}\Delta=\Delta r_{2i+1,2n}=\eta^{m(n-k)}\Delta r_{1,2i},\\
&r_{1,2n-2i-1}\Delta=\eta^{m(n-k-\frac{1}{2})}\Delta r_{1,2i+1}\rho_2=\eta^{m(n-k-1)}\Delta r_{1,2i+1}r^{-1}_{1,2n-1}.    
\end{align*}
  
\end{proof}

\begin{lem}\label{Lem:Atrels_boundary}
We have the following identities: $(i<2j+1)$
\[
\begin{aligned}
t'^{-1}_{i,2j+1}\AA_j&=\AA^{-1}_jT_jt'^{-1}_{i,2j+3},\\
t'^{-1}_{i,2j+2}T_j^{-1}\AA_j&=\AA_jT_j^{-1}t'^{-1}_{i,2j+2},\\
T_jt'^{-1}_{i,2j+1}t'^{-1}_{i,2j+2}T_j^{-1}\AA_j&=T_j\AA^{-1}_jt'^{-1}_{i,2j+3}t'^{-1}_{i,2j+2}.
\end{aligned}
\]    
\end{lem}
\begin{proof}
The first and third identities follow from straightforward calculations. We focus on the second equality, verifying the case for $i=1$; the remaining cases follow by similar arguments.
\[
\begin{aligned}
t'^{-1}_{1,2j+2}T_j^{-1}\AA_j&=\prod^{j-1}_{k=0}T'^{-1}_k\prod^{0}_{k=j-1}\BB_k(\prod^{j-1}_{k=0}\BB_k)T^{-1}_j\AA_k,\\
&=\prod^{j-1}_{k=0}T'^{-1}_k\prod^{0}_{k=j-1}\BB_k(\prod^{j-2}_{k=0}\BB_k)\underline{\BB_{j-1}T^{-1}_j\AA_{j}},\\
&=\prod^{j-1}_{k=0}T'^{-1}_k\underline{\BB_{j-1}\AA_{j}}T^{-1}_j\prod^{0}_{k=j-2}\BB_k(\prod^{j-1}_{k=0}\BB_k),\qquad (\text{Eq.~\eqref{eq:AB_rels}})\\
&=(\prod^{j-2}_{k=0}T'^{-1}_k)T'^{-1}_j\AA_j\BB_{j-1}T^{-1}_j\prod^{0}_{k=j-2}\BB_k(\prod^{j-1}_{k=0}\BB_k),\qquad (\text{Eq.~\eqref{eq:AB_rels}})\\
&=\AA_jT^{-1}_j(\prod^{j-2}_{k=0}T'^{-1}_k)T'^{-1}_{j-1}\BB_{j-1}\prod^{0}_{k=j-2}\BB_k(\prod^{j-1}_{k=0}\BB_k),\\
&=\AA_jT_j^{-1}t'^{-1}_{1,2j+2}.
\end{aligned}
\]
\end{proof}

\begin{prop}\label{prop:r_pentagon}
 The following relations hold:
 \begin{align}
 r^{-1}_{i,l-1}r_{i,k-1}r_{j,l-1}r_{k,m-1}r^{-1}_{j,m-1}=r^{-1}_{j,m-1}r_{k,m-1}r_{j,l-1}r_{i,k-1}r^{-1}_{i,l-1}.
 \end{align}
\end{prop}
\begin{proof}
The identity is equivalent to:
\[
\begin{aligned}
(\prod^{k}_{p=l-1}t'^{-1}_{i,p})r_{j,l-1}\prod^{k-1}_{p=j}t^{-1}_{p,m-1}&=(\prod^{k-1}_{p=j}t^{-1}_{p,m-1})r_{j,l-1}\prod^{k}_{p=l-1}t'^{-1}_{i,p},\\   
(\prod^{k}_{p=l-1}t'^{-1}_{i,p}\prod^{k}_{p=l-1}t'_{j,p})\prod^{k-1}_{p=j}t_{p,k-1}\prod^{k-1}_{p=j}t^{-1}_{p,m-1}&=\prod^{k-1}_{p=j}t^{-1}_{p,m-1}\prod^{k-1}_{p=j}t_{p,k-1}\prod^{k}_{p=l-1}t'_{j,p}\prod^{k}_{p=l-1}t'^{-1}_{i,p}.\\
\end{aligned}
\]
We can first assume $m=2n+1$, other cases will follow from conjugating by $F$. Using the commutation relation in Proposition \ref{prop:t_comm_rels}, we have: 
\[
\begin{aligned}
\prod^{k}_{p=l-1}t'^{-1}_{1,p}\prod^{k}_{p=l-1}t'_{j,p}&=\prod^{l-1}_{p=k}t'^{-1}_{1,p}t'_{j,p},\\
\prod^{k}_{p=l-1}t'_{j,p}\prod^{k}_{p=l-1}t'^{-1}_{1,p}&=\prod^{k}_{p=l-1}t'_{j,p}t'^{-1}_{1,p}.   
\end{aligned} 
\] 
From Lemma \ref{lem:rho_2tr_rels}, we have: 
\[
\prod^{k-1}_{p=j}t_{p,k-1}\prod^{k-1}_{p=j}t^{-1}_{p,2n}=r_{j,k-1}(\prod^{k-1}_{p=j}t'_{1,p})\rho_2^{\epsilon}.
\]
Hence, it suffices to prove the following:
\[
\prod^{l-1}_{p=k}t'^{-1}_{1,p}t'_{j,p}\prod^{k-1}_{p=j}t'_{1,p}=\prod^{k-1}_{p=j}t'_{1,p}\prod^{k}_{p=l-1}t'_{j,p}t'^{-1}_{1,p}.
\]
Now, by definition, we have:
\begin{align*}
t'^{-1}_{i,2s+1}t'_{2j+1,2s+1}&=(\prod^{j}_{k=s-1}\AA_k)t'^{-1}_{i,2j+1}\prod^{s-1}_{k=j}\AA^{-1}_k\\ 
t'^{-1}_{i,2s+1}t'_{2j,2s+1}&=(\prod^{j}_{k=s-1}\AA_k)t'^{-1}_{i,2j+1}T'_{j-1}\prod^{s-1}_{k=j}\AA^{-1}_k\\ t'^{-1}_{i,2s}t'_{2j+1,2s}&=(\prod^{j}_{k=s-2}\BB_k)t'^{-1}_{i,2j+2}T^{-1}_j\prod^{s-2}_{k=j}\BB^{-1}_k\\ 
t'^{-1}_{i,2s}t'_{2j+2,2s}&=(\prod^{j}_{k=s-2}\BB_k)t'^{-1}_{i,2j+2}\prod^{s-2}_{k=j}\BB^{-1}_k 
\end{align*}

We have the following Lemma:
\begin{lem}\label{lem:ttcomm_rel}
The following relations hold, the graphical interpretation of these relations is provided in Figure \ref{fig:ttcomm_rel_graph}.
\[
\begin{aligned}
t'^{-1}_{i,2s}t'_{j,2s}t'^{-1}_{i,2s+1}t'_{j,2s+1}&=T'^{-1}_{s-1}t'^{-1}_{i,2s+1}t'_{j,2s+1}T'_{s-1}t'^{-1}_{i,2s}t'_{j,2s},\\
t'^{-1}_{i,2s-1}t'_{j,2s-1}T'^{-1}_{s-1}t'^{-1}_{i,2s+1}t'_{j,2s+1}T'_{s-1}&=\AA^{-1}_{s-1}t'^{-1}_{i,2s-1}t'_{j,2s-1}\AA_{s-1}t'^{-1}_{i,2s-1}t'_{j,2s-1}.
\end{aligned}
\]
\end{lem}
\begin{figure}
    \centering
    \includegraphics[width=1\linewidth]{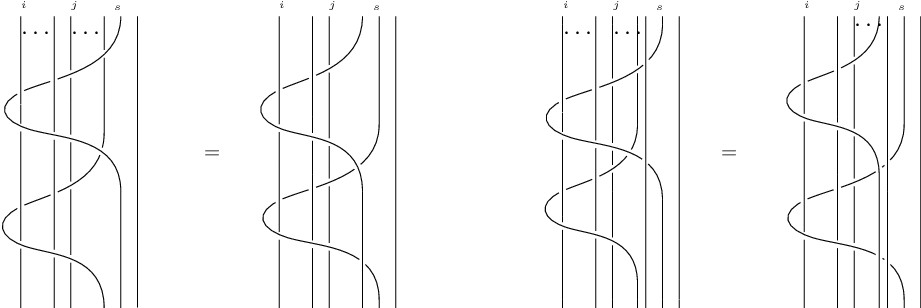}
    \caption{Graphic interpretation of Lemma \ref{lem:ttcomm_rel} }
    \label{fig:ttcomm_rel_graph}
\end{figure}
\begin{proof}
For the first identity, from Proposition \ref{prop:ABC_rels}, we have 
 \[
\BB^{-1}_{k-2}\BB^{-1}_{k-1}T_{k}\AA_{k-1}=T^{-1}_{k-1}\AA_{k-1}T_{k}\BB^{-1}_{k-2}T_{k-1}\BB^{-1}_{k-1},\quad \BB_{k-1}\BB_{k-2}T^{-1}_{k-1}\AA_{k-1}T_k=\AA_{k-1}\BB_{k-1}\BB_{k-2}.
 \]
and from Lemma \ref{Lem:Atrels_boundary}, we have
\[
t'^{-1}_{i,2s}t'_{j,2s}\AA_{s-1}=t'^{-1}_{i,2s}T^{-1}_{s-1}\AA_{s-1}t'_{j,2s}T_{s-1}= T'^{-1}_{s-1}\AA_{s-1}t'^{-1}_{i,2s}t'_{j,2s}T_{s-1}.
\]
Hence, we have the following (the case for $t'^{-1}_{i,2s}t'_{2j,2s}$ will be similar):  
\[
\begin{aligned}
t'^{-1}_{i,2s}t'_{2j+1,2s}\prod_{k=s-1}^{j}\AA_{k}&=T'^{-1}_{s-1}\AA_{s-1}t'^{-1}_{i,2s}t'_{2j+1,2s}T_{s-1}\prod_{k=s-2}^{j}\AA_{k},\\
&=T'^{-1}_{s-1}\AA_{s-1}\AA_{s-2}\BB_{s-2}t'^{-1}_{i,2s-2}t'_{2j+1,2s-2}T_{s-2}\BB^{-1}_{s-2}\prod_{k=s-3}^{j}\AA_{k},\\
&=\cdots,\\
&=T'^{-1}_{s-1}(\prod_{k=s-1}^{j+1}\AA_{k})\BB_{s-2}\cdots\BB_jt'^{-1}_{i,2j+2}T^{-1}_j\BB^{-1}_jT_{j+1}\cdots\BB^{-1}_{s-2}\AA_{j},\\
&=T'^{-1}_{s-1}(\prod_{k=s-1}^{j}\AA_{k})\BB_{s-2}\cdots\BB_jt'^{-1}_{i,2j+2}\BB^{-1}_j\cdots\BB^{-1}_{s-2}.
\end{aligned}
\]
which also implies: 
\[
(\prod_{k=j}^{s-1}\AA^{-1}_k)T'_{s-1}t'^{-1}_{i,2s}t'_{2j+1,2s}=\BB_{s-2}\cdots\BB_jt'^{-1}_{i,2j+2}\BB^{-1}_j\cdots\BB^{-1}_{s-2}\prod_{k=j}^{s-1}\AA^{-1}_k.
\]
Moreover, from Proposition \ref{prop:t_comm_rels}, we have: 
\[
[\BB_{s-2}\cdots\BB_jt'^{-1}_{i,2j+2}\BB^{-1}_j\cdots\BB^{-1}_{s-2},t'^{-1}_{i,2j+1}]=0.
\]
Therefore we have:
\[
\begin{aligned}
t'^{-1}_{i,2s}t'_{2j+1,2s}t'^{-1}_{i,2s+1}t'_{2j+1,2s+1}&=T'^{-1}_{s-1}(\prod_{k=s-1}^{j}\AA_{k})\BB_{s-2}\cdots\BB_jt'^{-1}_{i,2j+2}\BB^{-1}_j\cdots\BB^{-1}_{s-2}t'^{-1}_{i,2j+1}\prod_{k=j}^{s-1}\AA^{-1}_k\\
&=T'^{-1}_{s-1}(\prod_{k=s-1}^{j}\AA_{k})t'^{-1}_{i,2j+1}(\prod_{k=j}^{s-1}\AA^{-1}_k) T'_{s-1} t'^{-1}_{i,2s}t'_{2j+1,2s}\\
&=T'^{-1}_{s-1}t'^{-1}_{i,2s+1}t'_{2j+1,2s+1}T'_{s-1}t'^{-1}_{i,2s}t'_{2j+1,2s}    
\end{aligned}
\]
For the second one, by repeatedly applying Lemma \ref{lem:Abraid_rels}, we have 
\begin{align*}
 t'^{-1}_{i,2s-1}t'_{j,2s-1}T'^{-1}_{s-1}t'^{-1}_{i,2s+1}t'_{j,2s+1}T'_{s-1}=&\AA^{-1}_{s-1}\AA_{s-1}\AA_{s-2}t'^{-1}_{i,2s-3}t'_{j,2s-3}\underline{\AA^{-1}_{s-2}\AA_{s-1}T^{-1}_{s-1}\AA_{s-2}}\\
 &t'^{-1}_{i,2s-3}t'_{j,2s-3}\AA^{-1}_{s-2} T_{s-1}\AA^{-1}_{s-1},\\
 =&\AA^{-1}_{s-1}\AA_{s-1}\AA_{s-2}t'^{-1}_{i,2s-3}t'_{j,2s-3}\underline{\AA_{s-1}T^{-1}_{s-1}T^{-1}_{s-2}\AA_{s-2}T_{s-1}\AA^{-1}_{s-1}}\\
 &t'^{-1}_{i,2s-3}t'_{j,2s-3}\AA^{-1}_{s-2} T_{s-1}\AA^{-1}_{s-1},\\
 =&\AA^{-1}_{s-1}\underline{\AA_{s-1}\AA_{s-2}\AA_{s-1}T^{-1}_{s-1}}t'^{-1}_{i,2s-3}t'_{j,2s-3}T^{-1}_{s-2}\AA_{s-2}\\
 &t'^{-1}_{i,2s-3}t'_{j,2s-3}\underline{T_{s-1}\AA^{-1}_{s-1}\AA^{-1}_{s-2} T_{s-1}\AA^{-1}_{s-1}},\\
 =&\AA^{-1}_{s-1}\underline{\AA_{s-2}\AA_{s-1}T^{-1}_{s-2}\AA_{s-2}}t'^{-1}_{i,2s-3}t'_{j,2s-3}T^{-1}_{s-2}\AA_{s-2}\\
 &t'^{-1}_{i,2s-3}t'_{j,2s-3}\underline{\AA^{-1}_{s-2}T_{s-2}\AA^{-1}_{s-1} T'_{s-1}\AA^{-1}_{s-2}},\\
 =&\AA^{-1}_{s-1}\AA_{s-2}\AA_{s-1}T^{-1}_{s-2}\AA_{s-2}\AA_{s-3}t'^{-1}_{i,2s-5}t'_{j,2s-5}\underline{\AA^{-1}_{s-3}T^{-1}_{s-2}\AA_{s-2}\AA_{s-3}}\\
 &t'^{-1}_{i,2s-5}t'_{j,2s-5}\AA^{-1}_{s-3}\AA^{-1}_{s-2}T_{s-2}\AA^{-1}_{s-1} T'_{s-1}\AA^{-1}_{s-2},\\
 =&\AA^{-1}_{s-1}\prod^{s-3}_{k=s-2}\AA_{k}\prod^{s-2}_{k=s-1}\AA_{k}T^{-1}_{s-3}\AA_{s-3}t'^{-1}_{i,2s-5}t'_{j,2s-5}T^{-1}_{s-3}\AA_{s-3}\\
 &t'^{-1}_{i,2s-5}t'_{j,2s-5}\AA^{-1}_{s-3}T_{s-3}(\prod_{k=s-2}^{s-1}\AA^{-1}_{k})T'_{s-1}\prod^{s-2}_{k=s-3}\AA^{-1}_{k}.
\end{align*}
Applying the relation established in Lemma \ref{lem:Abraid_rels}, the left-hand side reduces to the following expression:
\[
\begin{aligned}
 &\AA^{-1}_{s-1}\prod^{j}_{k=s-2}\AA_{k}\prod^{j+1}_{k=s-1}\AA_{k}T^{-1}_{j}\AA_{j}t'^{-1}_{i,2j+1}T^{-1}_{j}\AA_{j}
 t'^{-1}_{i,2j+1}\AA^{-1}_{j}T_{j}(\prod_{k=j+1}^{s-1}\AA^{-1}_{k})T'_{s-1}\prod^{s-2}_{k=j}\AA^{-1}_{k},\\
 =&\AA^{-1}_{s-1}\prod^{j}_{k=s-2}\AA_{k}\prod^{j+1}_{k=s-1}\AA_{k}t'^{-1}_{i,2j+3}
 t'^{-1}_{i,2j+1}\AA^{-1}_{j}(\prod_{k=j+1}^{s-1}\AA^{-1}_{k})T'_{s-1}\prod^{s-2}_{k=j}\AA^{-1}_{k},\\
 =&\AA^{-1}_{s-1}\prod^{j}_{k=s-2}\AA_{k}t'^{-1}_{i,2j+1}\prod^{j+1}_{k=s-1}\AA_{k}t'^{-1}_{i,2j+3}
 \AA^{-1}_{j}(\prod_{k=j+1}^{s-1}\AA^{-1}_{k})T'_{s-1}\prod^{s-2}_{k=j}\AA^{-1}_{k},\\
 =&\AA^{-1}_{s-1}\prod^{j}_{k=s-2}\AA_{k}t'^{-1}_{i,2j+1}\prod^{j+1}_{k=s-1}\AA_{k}T_j^{-1}\AA_{j}
 (\prod_{k=j+1}^{s-1}\AA^{-1}_{k})T'_{s-1}t'^{-1}_{i,2j+1}\prod^{s-2}_{k=j}\AA^{-1}_{k}.
\end{aligned}
\]
Again, we use relation in Lemma \ref{lem:Abraid_rels} to rewrite the middle part,
\[
\begin{aligned}
\prod^{j+1}_{k=s-1}\AA_{k}T_j^{-1}\AA_{j}
 (\prod_{k=j+1}^{s-1}\AA^{-1}_{k})T'_{s-1}&=\AA^{-1}_{j}\prod^{j+2}_{k=s-1}\AA_{k}T_{j+1}^{-1}\AA_{j+1}\AA_{j}
 (\prod_{k=j+2}^{s-1}\AA^{-1}_{k})T'_{s-1},\\
 &=\cdots,\\
 &=(\prod_{k=j}^{s-2}\AA^{-1}_{k})T^{-1}_{s-1}\AA_{s-1}\prod^{j}_{k=s-2}\AA_{k}T'_{s-1},\\
 &=(\prod_{k=j}^{s-2}\AA^{-1}_{k})\AA_{s-1}\prod^{j}_{k=s-2}\AA_{k}.
\end{aligned}
\]
Plug it in, we have:

\begin{align*}
 t'^{-1}_{i,2s-1}t'_{2j+1,2s-1}T'^{-1}_{s-1}t'^{-1}_{i,2s+1}t'_{2j+1,2s+1}T'_{s-1}&=\AA^{-1}_{s-1}\prod^{j}_{k=s-2}\AA_{k}t'^{-1}_{i,2j+1}(\prod_{k=j}^{s-2}\AA^{-1}_{k})\AA_{s-1}\prod^{j}_{k=s-2}\AA_{k}t'^{-1}_{i,2j+1}\prod^{s-2}_{k=j}\AA^{-1}_{k},\\
 &=\AA^{-1}_{s-1}t'^{-1}_{i,2s-1}t'_{2j+1,2s-1}\AA_{s-1}t'^{-1}_{i,2s-1}t'_{2j+1,2s-1}.
\end{align*}

\end{proof}

Now from Lemma \ref{lem:ttcomm_rel}, the term $t'^{-1}_{1,l}t'_{j,l}$ can be moved to the front of the product $\prod^{l}_{p=k}t'^{-1}_{1,p}t'_{j,p}$, as shown in the following equalities:
\[
\begin{aligned}
\prod^{2l+1}_{p=2k+1}t'^{-1}_{i,p}t'_{j,p}=&(\prod ^{k}_{s=l-1}\AA^{-1}_s)t'^{-1}_{i,2k+1}t'_{j,2k+1}\prod ^{l-1}_{s=k}\AA_s\prod^{2l}_{p=2k+1}t'^{-1}_{i,p}t'_{j,p},\\  
\prod^{2l+1}_{p=2k}t'^{-1}_{i,p}t'_{j,p}=&(\prod ^{k}_{s=l-1}\AA^{-1}_s)T'^{-1}_{k-1}t'^{-1}_{i,2k+1}t'_{j,2k+1}T'_{k-1}\prod ^{l-1}_{s=k}\AA_s\prod^{2l}_{p=2k}t'^{-1}_{i,p}t'_{j,p},\\
\prod^{2l}_{p=2k+1}t'^{-1}_{i,p}t'_{j,p}=&(\prod ^{k}_{s=l-2}\BB^{-1}_s)T^{-1}_kt'^{-1}_{i,2k+2}t'_{j,2k+2}T_k\prod ^{l-2}_{s=k}\BB_s\prod^{2l-1}_{p=2k+1}t'^{-1}_{i,p}t'_{j,p},\\  
\prod^{2l}_{p=2k}t'^{-1}_{i,p}t'_{j,p}=&(\prod ^{k-1}_{s=l-2}\BB^{-1}_s)t'^{-1}_{i,2k}t'_{j,2k}\prod^{l-2}_{s=k-1}\BB_s\prod^{2l-1}_{p=2k}t'^{-1}_{i,p}t'_{j,p}.
\end{aligned}
\]
Note that the last two can be obtained from conjugating the first two by $F$.

Next, we have:
\begin{align*}
t'^{-1}_{i,2k+1}t'_{2j+1,2k+1}\prod^{2k}_{p=2j+1}t'_{1,p}=&(\prod^{j}_{p=k-1}\AA_p)t'^{-1}_{i,2j+1}\prod^{k-1}_{p=j}\AA^{-1}_p \prod^{2k}_{p=2j+1}t'_{1,p},\\
=&(\prod^{j}_{p=k-1}\AA_p)t'^{-1}_{i,2j+1}(\prod^{k-2}_{p=j}\AA^{-1}_p)T'^{-1}_{k-1}t'_{1,2k+1}t'_{1,2k}(\prod^{2k-2}_{p=2j+1}t'_{1,p})\AA_{k-1}, \qquad \text{( Lem.~\ref{Lem:Atrels_boundary})}\\
=&\cdots,\\
=&(\prod^{j}_{p=k-1}\AA_p)t'^{-1}_{i,2j+1}\prod^{k-1}_{p=j}T'^{-1}_p \prod^{2k+1}_{p=2j+2}t'_{1,p}\prod^{k-1}_{p=j}\AA_p, \\
=&(\prod^{j}_{p=k-1}\AA_p)\prod^{k-1}_{p=j}T'^{-1}_p \prod^{2k+1}_{p=2j+2}t'_{1,p}t'^{-1}_{i,2j+1}\prod^{k-1}_{p=j}\AA_p, \qquad \text{( Prop.~\ref{prop:t_comm_rels})}\\
=&\prod^{2k}_{p=2j+1}t'_{1,p}(\prod^{j}_{p=k-1}\AA^{-1}_p)t'^{-1}_{i,2j+1}\prod^{k-1}_{p=j}\AA_p,\qquad \text{( Lem.~\ref{Lem:Atrels_boundary})}\\
=&\prod^{2k}_{p=2j+1}t'_{1,p}t'_{2j+1,2k+1}t'^{-1}_{i,2k+1}.
\end{align*}
Similarly, we have:
\begin{align*}
t'^{-1}_{i,2k+1}t'_{2j,2k+1}\prod^{2k}_{p=2j}t'_{1,p}=&(\prod^{j}_{p=k-1}\AA_p)t'^{-1}_{i,2j+1}T'^{-1}_{j-1}\prod^{k-1}_{p=j}\AA^{-1}_p \prod^{2k}_{p=2j}t'_{1,p},\\
=&(\prod^{j}_{p=k-1}\AA_p)t'^{-1}_{i,2j+1}T'^{-1}_{j-1}t'_{1,2j}\prod^{k-1}_{p=j}T'^{-1}_p \prod^{2k+1}_{p=2j+2}t'_{1,p}\prod^{k-1}_{p=j}\AA_p, \qquad \text{( Lem.~\ref{Lem:Atrels_boundary})}\\
=&(\prod^{j}_{p=k-1}\AA_p)\prod^{k-1}_{p=j}T'^{-1}_p (\prod^{2k+1}_{p=2j+2}t'_{1,p})t'^{-1}_{i,2j+1}T'^{-1}_{j-1}t'_{1,2j}\prod^{k-1}_{p=j}\AA_p, \qquad \text{( Prop.~\ref{prop:t_comm_rels})} \\
=&\prod^{2k}_{p=2j+1}t'_{1,p}(\prod^{j}_{p=k-1}\AA^{-1}_p)t'^{-1}_{i,2j+1}T'^{-1}_{j-1}t'_{1,2j}\prod^{k-1}_{p=j}\AA_p, \qquad \text{( Lem.~\ref{Lem:Atrels_boundary})}\\
=&\prod^{2k}_{p=2j+1}t'_{1,p}(\prod^{j}_{p=k-1}\AA^{-1}_p)T^{-1}_{j-1}\AA_{j-1}t'^{-1}_{i,2j-1}\AA_{j-1}T'^{-1}_{j-1}t'_{1,2j}\prod^{k-1}_{p=j}\AA_p,\\
=&\prod^{2k}_{p=2j+1}t'_{1,p}(\prod^{j}_{p=k-1}\AA^{-1}_p)t'_{1,2j}T'^{-1}_{j-1}\AA_{j-1}\ t'^{-1}_{i,2j-1}\AA_{j-1}T^{-1}_{j-1}\prod^{k-1}_{p=j}\AA_p, \qquad \text{( Lem.~\ref{Lem:Atrels_boundary})}\\
=&\prod^{2k}_{p=2j}t'_{1,p}(\prod^{j}_{p=k-1}\AA^{-1}_p)T'^{-1}_{j-1}t'^{-1}_{i,2j+1}\prod^{k-1}_{p=j}\AA_p, \\
=&\prod^{2k}_{p=2j}t'_{1,p}t'_{2j,2k+1}t'^{-1}_{i,2k+1}.   
\end{align*}
We also have:
\[
\begin{aligned}
t'^{-1}_{i,2k+1}t'_{j,2k+1}t'_{1,2k-1}=t'_{1,2k-1}T'_{k-1}\AA^{-1}_{k-1}t'^{-1}_{i,2k-1}t'_{j,2k-1}\AA_{k-1}T'^{-1}_{k-1}.
\end{aligned}
\]
Therefore, we have:
\[
t'^{-1}_{i,2k+1}t'_{j,2k+1}\prod^{2k-1}_{p=j}t'_{1,p}=(\prod^{2k-1}_{p=j}t'_{1,p})T'_{k-1}t'_{j,2k+1}t'^{-1}_{i,2k+1}T'^{-1}_{k-1}.
\]
Now we can do induction, suppose: 
\[
\prod^{l-1}_{p=k}t'^{-1}_{1,p}t'_{j,p}\prod^{k-1}_{p=j}t'_{1,p}=\prod^{k-1}_{p=j}t'_{1,p}\prod^{k}_{p=l-1}t'_{j,p}t'^{-1}_{i,p}.
\]
We need to show:
\[
\prod^{l}_{p=k}t'^{-1}_{1,p}t'_{j,p}\prod^{k-1}_{p=j}t'_{1,p}=\prod^{k-1}_{p=j}t'_{1,p}\prod^{k}_{p=l}t'_{j,p}t'^{-1}_{1,p}.    
\]
We have (The computation for the product $\prod^{2l+1}_{p=2k+1}t'^{-1}_{i,p}t'_{j,p}\prod^{2k}_{p=j}t'_{1,p}$ follows an analogous, simpler, argument; thus, we omit the details here):
\[
\begin{aligned}
\prod^{2l+1}_{p=2k}t'^{-1}_{i,p}t'_{j,p}\prod^{2k-1}_{p=j}t'_{1,p}=&(\prod ^{k}_{s=l-1}\AA^{-1}_s)T'^{-1}_{k-1}t'^{-1}_{i,2k+1}t'_{j,2k+1}T'_{k-1}\prod ^{l-1}_{s=k}\AA_s\prod^{2l}_{p=2k}t'^{-1}_{i,p}t'_{j,p}\prod^{2k-1}_{p=j}t'_{1,p},\qquad(\text{Lem.~\ref{lem:ttcomm_rel}})\\
=&(\prod ^{k}_{s=l-1}\AA^{-1}_s)T'^{-1}_{k-1}t'^{-1}_{i,2k+1}t'_{j,2k+1}T'_{k-1}\prod ^{l-1}_{s=k}\AA_s\prod^{2k-1}_{p=j}t'_{1,p} \prod^{2k}_{p=2l}t'_{j,p}t'^{-1}_{i,p},\qquad \text{(By induction)}\\
=&(\prod ^{k}_{s=l-1}\AA^{-1}_s)T'^{-1}_{k-1}\underline{t'^{-1}_{i,2k+1}t'_{j,2k+1}(\prod^{2k-1}_{p=j}t'_{1,p})}T'_{k-1}\prod ^{l-1}_{s=k}\AA_s \prod^{2k}_{p=2l}t'_{j,p}t'^{-1}_{i,p},\qquad \text{(Prop. \ref{prop:t_comm_rels})}\\
=&\prod^{2k-1}_{p=j}t'_{1,p}(\prod ^{k}_{s=l-1}\AA^{-1}_s)t'_{j,2k+1}t'^{-1}_{i,2k+1}\prod ^{l-1}_{s=k}\AA_s \prod^{2k}_{p=2l}t'_{j,p}t'^{-1}_{i,p}, \qquad \text{(Lem. \ref{Lem:Atrels_boundary})}\\
=&\prod^{2k-1}_{p=j}t'_{1,p}\prod^{2k}_{p=2l+1}t'_{j,p}t'^{-1}_{i,p}.
\end{aligned}
\]
The remaining case,  $\prod^{2l}_{p=k}t'^{-1}_{i,p}t'_{j,p}\prod^{k-1}_{p=j}t'_{1,p}$, follows directly by conjugating the previous identities by $F$.
\end{proof}

\subsection{Projective representations of $\SMod(\Sigma_{(n-1)(m-1)})$}

We fix our notation for the relevant surface mapping class groups. Let $\Sigma_{g,n}$ denote a closed oriented surface of genus $g$ with $n$ punctures. We consider the balanced superelliptic covering $\Sigma_{(n-1)(k-1)} \to \Sigma_0$ of degree $k > 2$ with $2n$ branch points. Let $\LMod(\Sigma_{0,2n})$ be the associated liftable mapping class group, and let $\SMod(\Sigma_{(n-1)(k-1)})$ denote the balanced superelliptic mapping class group. The presentation for $\SMod(\Sigma_{(n-1)(k-1)})$ is given in \cite[Thm.~6.12]{HO25}; see also \cite{GW17lifting, GW17blifting}.With these definitions in place, we now proceed to the proof of the following theorem:

\begin{Thm}\label{thm:Smod_rep}
 The configuration space $Conf(\CC)_{m,n}$ together with operators $\AA_j, \BB_j,  r_{i,j},\Delta,$ gives a unitary projective representation of $\SMod(\Sigma_{(n-1)(m-1)})$.    
\end{Thm}

\begin{proof}
We define the map
\[
\rho:\SMod(\Sigma_{(n-1)(m-1)})\to \PGL(Conf(\CC)_{m,n}),
\]
by $\rho(\tilde{h}_{2i+1})=\AA_i$, $\rho(\tilde{h}_{2i+2})=\BB_i$, $\rho(\tilde{t}_{i,j})=r_{i,j}$, $\rho(r)=\Delta$.

This is well defined follows from relations in Proposition \ref{prop:ABC_rels},\ref{prop:t_comm_rels}, \ref{prop:r_comm_rels}, \ref{prop:TAB_rels},  \ref{prop:r_pentagon}, \ref{prop:r_delta_rels}, Lemma \ref{lem:Abraid_rels}, and relations \eqref{delta_square_to_1}, \eqref{r_powerthree_to_1}. 
\end{proof}

Let $B_n$ be the braid group on $n$-strands, and $PB_n$ be the pure braid group. Let $\Mod(\Sigma_{0,2n})$ be the spherical braid group, and $\PMod(\Sigma_{0,2n})$ be the spherical pure braid group, $q$ be the following map:
\begin{equation}
    B_{2n}\xrightarrow[]{q}\operatorname{Mod}(\Sigma_{0,2n}).
\end{equation}
Recall that \(W_{2n}\) and \(B_{2n}^{\mathrm{par}}\) denote the subgroups of parity-compatible permutations and braids, respectively; we have the following commutative diagram:
\[
\begin{tikzcd}[row sep=large, column sep=large]
\ker(q)
  \arrow[r, "\cong"] 
  \arrow[d, hook] 
&
\ker(q)
  \arrow[d, hook] 
&
\phantom{A_{13}}
\\
PB_{2n} 
  \arrow[r, hook ] 
  \arrow[d, two heads, "q"'] 
&
B_{2n}^{\mathrm{par}}
  \arrow[r, two heads ] 
  \arrow[d, two heads, "q"] 
&
W_{2n}
  \arrow[d, "\cong"]
\\
\operatorname{PMod(\Sigma_{0,2n})} 
  \arrow[r, hook] 
&
\LMod(\Sigma_{0,2n}) 
  \arrow[r, two heads] 
&
W_{2n}
\end{tikzcd}
\]

Therefore, using the same technique in \cite{GW17lifting, HO25}, we can get a presentation for the group $G\subset B_{2n}$. Using relations that do not involve $\eta$, i.e., Proposition \ref{prop:ABC_rels}, \ref{prop:TAB_rels}, \ref{prop:t_comm_rels}, \ref{prop:r_comm_rels}, \ref{prop:r_pentagon}, the first identity in Proposition \ref{prop:r_delta_rels}, Lemma \ref{lem:Delta_ABT_rels} and relations \eqref{delta_square_to_1}, we obtain the linear representation for the group $B_{2n}^{\mathrm{par}}$:

\begin{Cor}
The map $\rho$ induces a unitary representation of the $G\subset B_n$, therefore, given two isotopic braids $\beta_1,\beta_2\in B_{2n}$ which are parity compatible. We have: 
\[\rho(\beta_1)=\rho(\beta_2).\]
Or graphically, we have:
\begin{equation}\label{eq:braid_topological_equvalence}
\includegraphics[width=150pt]{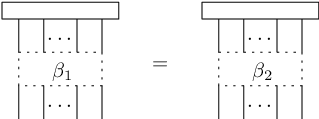} 
\end{equation}
\end{Cor}

The next relation follows directly from Proposition \ref{prop:ABC_rels}, or Figure \ref{fig:connect_two_braidings}.

\begin{equation}\label{eq:double_string_box_through}
\includegraphics[width=400pt]{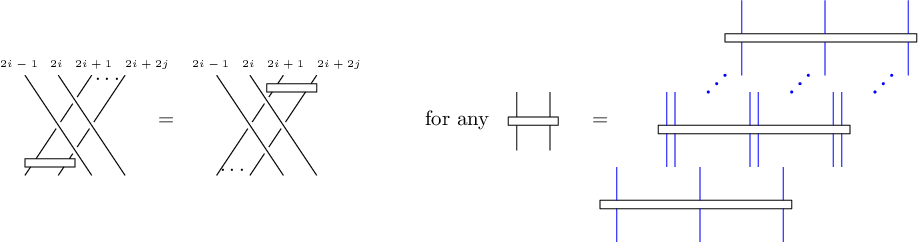}    
\end{equation}

\section{Applications}
In this section, we first examine the interplay between the maps $\phi_k,\iota_k$ and the braiding structures established in previous sections. Utilizing these relations, we derive the structure constants for the 2-box convolution product of the planar algebra, in terms of a generalized Verlinde formula. 

\subsection{Relations involving $\phi_k,\iota_k$ and the braiding structures}
\begin{lem}
The following identity holds
\begin{align*}
\phi_{2k+4}T_{k+1}\mathcal{A}^{-1}_{k+1}\iota_{2k+1}=\mathcal{B}_k&=\phi_{2k+1}\mathcal{A}^{-1}_{k+1}T_{k+1}\iota_{2k+4},\\
\phi_{2k+3}T'_{k}\mathcal{B}^{-1}_{k}\iota_{2k+1}=\mathcal{A}_k&=\phi_{2k}\mathcal{B}^{-1}_{k}T'_{k}\iota_{2k+3},\\
\end{align*}
or graphically
\begin{equation}\label{eq:phi_iota_AB_rels}
\includegraphics[width=220pt]{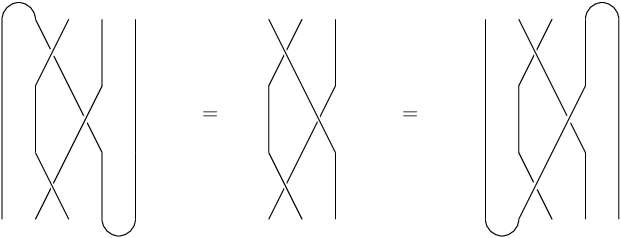}    
\end{equation}
\end{lem}

\begin{proof}
 The proof of the left-hand side is given in Figure \ref{fig:ABcapcup_rels}. Conjugating the first identity by $F$, we get the second identity
 \begin{figure}
     \centering     \includegraphics[width=0.8\linewidth]{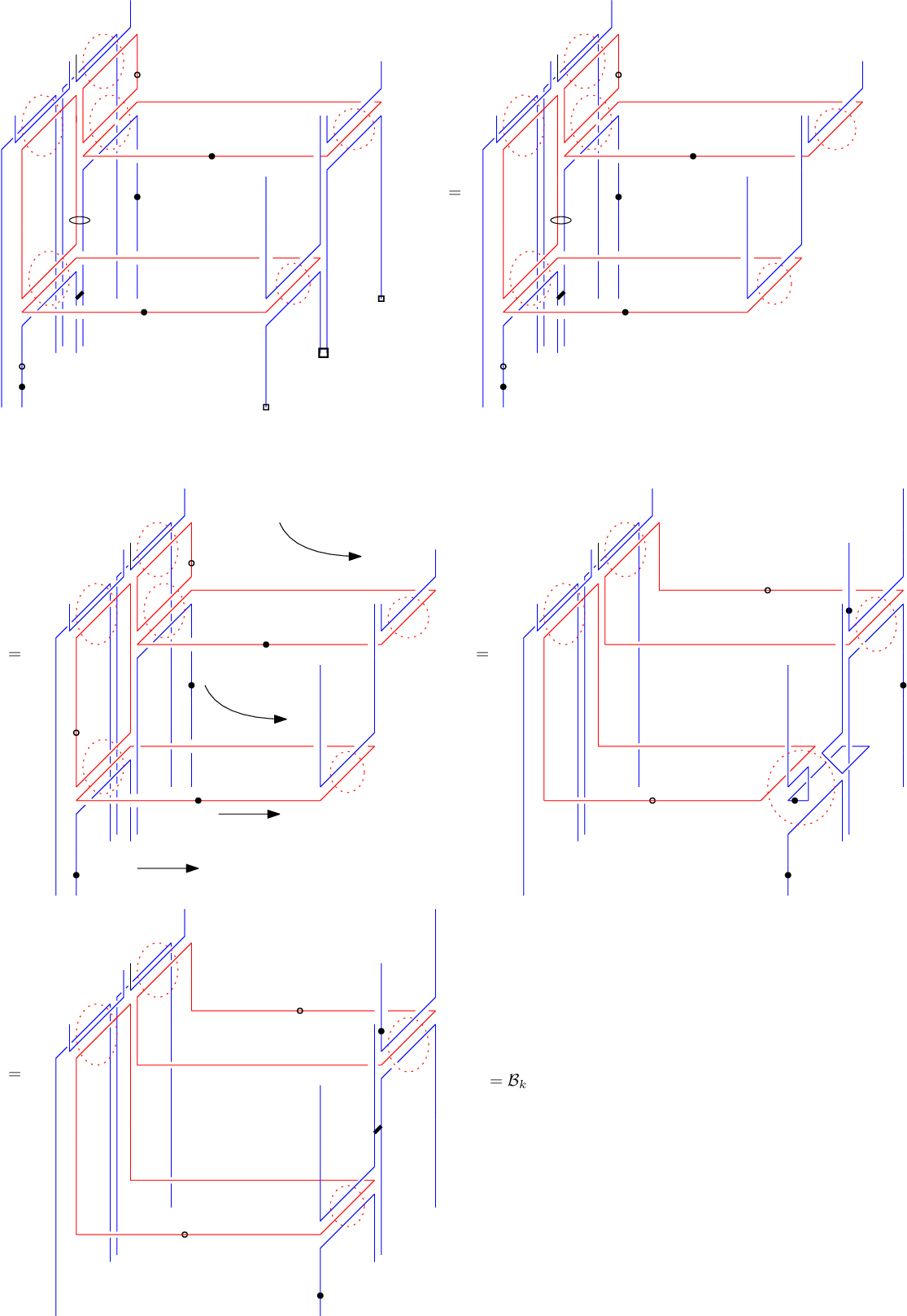}
     \caption{Proof for $(T_{k+1}\otimes\phi_{2k+4})\overline{\mathcal{A}}_{k+1}\iota_{2k+1}=\mathcal{B}_k$}
     \label{fig:ABcapcup_rels}
\end{figure}
\end{proof}

\begin{rmk}
 The relation \eqref{eq:phi_iota_AB_rels} is equivalent to the following, by applying the Zigzag relation of $\phi_k,\iota_k$:
 \[
 \includegraphics[width=300pt]{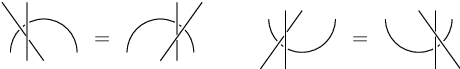}
 \]
\end{rmk}
\begin{lem}\label{lem:phi_F_and_far_comm}
The following relations holds (same for $\iota_k$).
\begin{equation}
 \includegraphics[width=400pt]{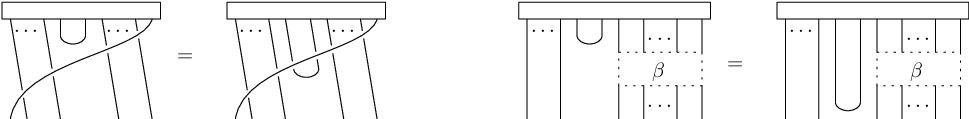}   
\end{equation}
\end{lem}
\begin{proof}
 The first identity is \cite[Prop.~6.10]{LX19}, also compare to \cite[Prop.~4.8]{LR24}. The identity for $\phi_{2k}$ is straightforward from the definition. the case for     $\phi_{2k}$ then follows from conjugating by $F$.
\end{proof}

\begin{lem}\label{lem:u_F_phi_iota_rels}
 We have the following relation in $Cf_{m,1}(\CC)$.
 \begin{align}
 \iota_0= \eta^{\frac{m}{4}}u^{-1}\iota_0=\eta^{-\frac{m}{4}}\phi_0F^{-1}\iota_0\iota_1,     
 \end{align}

or graphically
\[
\includegraphics[width=200pt]{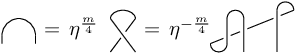}
\]
\end{lem}
\begin{proof}
These identities are straightforward from the graphic calculus in the configuration space.
\end{proof}

\subsection{Structure constant}
We now proceed to compute the structure constants of the 2-box convolution product for the $m$-interval Jones-Wassermann subfactor planar algebra. The case $m=2$ has been discussed extensively in \cite{LR24}, where the corresponding structure constants can be derived via standard Verlinde formula. We will focus here on $m=3$, and the case for $m>3$ proceed analogously.

We first fix an orthogonal basis for $\Hom(\mathbbm{1},XYZ)$ for any $X,Y,Z\in \Irr(\CC)$, denoted again by $ONB(XYZ)$ or simply $ONB$. 

\begin{defn}
Given any $\alpha_1,\alpha_2\in ONB(X_1Y_1Z_1)$, we define the following operator and denote it by $\tilde{p}_{\alpha_1\alpha_2}$, the corresponding vector in $Conf(\CC)_{3,2}$ is denoted by $p_{\alpha_1\alpha_2}$. 

\[
\includegraphics[width=250pt]{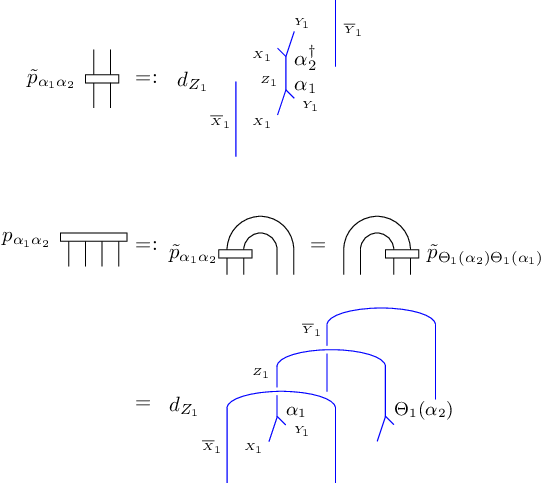}
\]
\end{defn}
Note $p_{\alpha_1\alpha_2}$ is a projection iff $\alpha_1=\alpha_2$, and it is straightforward that
\[
\sum_{X,Y,Z\in\Irr(\CC), \alpha\in ONB(XYZ)}\tilde{p}_{\alpha\alpha}=\Id_{Conf(\CC)_{3,n}}
\]
and we have $d^2_\alpha=:<p_{\alpha_1\alpha_2},p_{\alpha_1\alpha_2}>=d_{X_1}d_{Y_1}d_{Z_1}$, therefore, $\{\frac{1}{d_{\alpha}}p_{\alpha_1\alpha_2}\}_{\alpha_1,\alpha_2\in ONB(XYZ),X,Y,Z\in\Irr(\CC)}$ forms an orthonormal basis for $Conf(\CC)_{3,2}$, therefore we define the following unitary matrix $(L_{\alpha_1\alpha_2,\beta_1\beta_2})_{\alpha_i,\beta_i\in ONB}$ induced by the Fourier transformation $F$ on $Conf(\CC)_{3,2}$,
\[
L_{\alpha_1\alpha_2,\beta_1\beta_2}=:<F(\frac{1}{d_{\alpha}}p_{\alpha_1\alpha_2}),\frac{1}{d_{\beta}}p_{\beta_1\beta_2}>=LL(\frac{1}{d_{\alpha}}p_{\alpha_1\alpha_2},\Theta_2(\frac{1}{d_{\beta}}p_{\beta_1\beta_2}))
\]
The following lemma is immediately,
\begin{lem}
 Let $\alpha_1,\alpha_2,\alpha_3,\alpha_4\in ONB(XYZ)$, $\beta_1,\beta_2\in ONB(X'Y'Z')$ with $\delta_{XX'}\delta_{YY'}\delta_{ZZ'}=0$ we have 
 \begin{align*} \tilde{p}_{\alpha_1\alpha_2}\tilde{p}_{\alpha_3\alpha_4}&=\delta_{\alpha_2\alpha_3}\tilde{p}_{\alpha_1\alpha_4} \\
\tilde{p}_{\alpha_1\alpha_2}\tilde{p}_{\beta_1\beta_2}&=0
 \end{align*}
\end{lem}
\begin{defn}
 Let $\alpha_1,\alpha_2\in ONB(X_1Y_1Z_1),\ \beta_1,\beta_2\in ONB(X_2Y_2Z_2)$, we define the following convolution product of $p_{\alpha_1\alpha_2}$ and  $p_{\beta_1\beta_2}$, 
 \[
 \includegraphics[width=150pt]{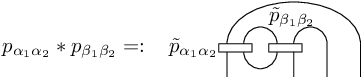}
 \]
\end{defn}

\begin{defn}
 Given $\alpha_1,\alpha_2,\beta_1,\beta_2,\gamma_1,\gamma_2\in ONB$, We denote the structure constant of the convolution product by $N^{\gamma_1\gamma_2}_{\alpha_1\alpha_2,\beta_1\beta_2}$, i.e.
\[
\frac{1}{d_\alpha}p_{\alpha_1\alpha_2}*\frac{1}{d_\beta}p_{\beta_1\beta_2}=\sum_{\gamma_1,\gamma_2\in ONB(XYZ), X,Y,Z\in \Irr(\CC)}N^{\gamma_1\gamma_2}_{\alpha_1\alpha_2,\beta_1\beta_2}\frac{1}{d_\gamma}p_{\gamma_1\gamma_2}
\]
\end{defn}

\begin{Thm}\label{thm:Verlinde_formula}
The structure constant $N^{\gamma_1\gamma_2}_{\alpha_1\alpha_2,\beta_1\beta_2}$ can be derived by the following generalized Verlinde formula  
\[
N^{\gamma_1\gamma_2}_{\alpha_1\alpha_2,\beta_1\beta_2}=\sum_{X,Y,Z\in \Irr(\CC)}\sum_{\lambda_1,\lambda_2,\lambda_3\in ONB(XYZ)}\frac{\overline{L_{\lambda_1\lambda_2,\alpha_1\alpha_2}}\overline{L_{\lambda_2\lambda_3,\beta_1\beta_2}}\overline{L_{\lambda_3\lambda_1,\Theta_1(\gamma_1)\Theta_1(\gamma_2)}}}{d_{\lambda}} 
\]
\end{Thm}
\begin{proof}
First of all, by using the planar isotopy identities we have the following:
\[
\includegraphics[width=400pt]{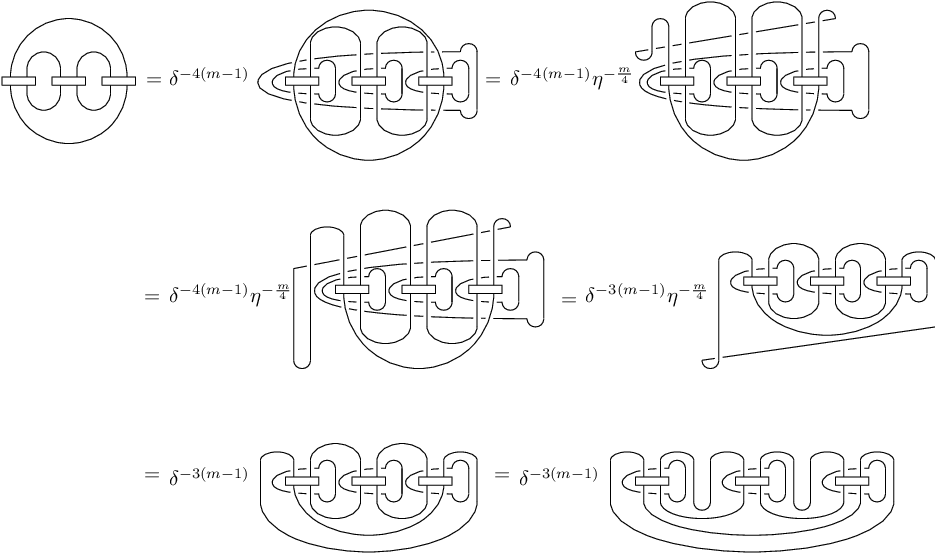}
\]
The first identity follows from relation \eqref{eq:double_string_box_through}, second identity is Lemma \ref{lem:u_F_phi_iota_rels} and \ref{lem:phi_F_and_far_comm}, the third identity is Lemma \ref{lem:phi_F_and_far_comm}, the forth identity is again relation \eqref{eq:double_string_box_through}, the fifth identity is again Lemma \ref{lem:u_F_phi_iota_rels} and the last identity follows from the Zigzag relation.

then we zoom in and observe the following
\[
\includegraphics[width=400pt]{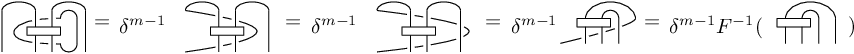}
\]
Here the first three equalities follow from relation \eqref{eq:double_string_box_through}, \eqref{eq:braid_topological_equvalence} and Lemma \ref{lem:phi_F_and_far_comm}, \ref{lem:u_F_phi_iota_rels} respectively. 

Therefore, we have the following
\[
\includegraphics[width=400pt]{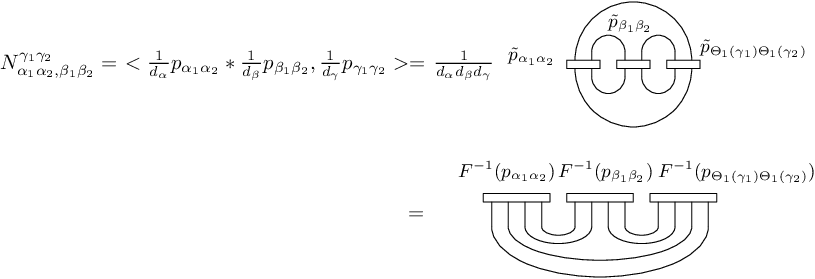}
\]

Now, since $<F^{-1}(p_{\alpha_1\alpha_2}),\frac{1}{\sqrt{d_\lambda}}p_{\lambda_1\lambda_2}>=d_\alpha<F^{-1}\frac{1}{\sqrt{d_\alpha}}p_{\alpha_1\alpha_2},\frac{1}{\sqrt{d_\lambda}}p_{\lambda_1\lambda_2}>=d_a\overline{L_{\lambda_1\lambda_2,\alpha_1\alpha_2}}$, evaluating the right-hand side, we have
\begin{align*}
N^{\gamma_1\gamma_2}_{\alpha_1\alpha_2,\beta_1\beta_2}&=\frac{1}{d_\alpha d_\beta d_\gamma}\sum_{X,Y,Z\in \Irr(\CC)}\sum_{\lambda_1,\lambda_2,\lambda_3\in ONB(XYZ)} \frac{d_\alpha d_\beta d_\gamma}{(d_{\lambda})^3}\overline{L_{\lambda_1\lambda_2,\alpha_1\alpha_2}}\overline{L_{\lambda_2\lambda_3,\beta_1\beta_2}}\overline{L_{\lambda_3\lambda_1,\Theta_1(\gamma_1)\Theta_1(\gamma_2)}}d^2_\lambda\\
&=\sum_{X,Y,Z\in \Irr(\CC)}\sum_{\lambda_1,\lambda_2,\lambda_3\in ONB(XYZ)}\frac{\overline{L_{\lambda_1\lambda_2,\alpha_1\alpha_2}}\overline{L_{\lambda_2\lambda_3,\beta_1\beta_2}}\overline{L_{\lambda_3\lambda_1,\Theta_1(\gamma_1)\Theta_1(\gamma_2)}}}{d_{\lambda}} 
\end{align*}
\end{proof}

\bibliographystyle{abbrv}
\bibliography{Reference}
\end{document}